\documentclass[11pt]{article}
\usepackage{graphicx,subfig}
\usepackage[T1]{fontenc}
\usepackage{lmodern} 
\usepackage{comment}
\usepackage{amsmath, amsthm, amssymb}
\usepackage{pifont}
\usepackage{centernot}
\usepackage{enumitem}
\usepackage[english]{babel}
\usepackage[autostyle]{csquotes}
\usepackage{wrapfig}
\usepackage[a4paper, left=1in, right=1in, top=1in, bottom=1in]{geometry}
\usepackage{microtype}
\usepackage{setspace}
\usepackage{xcolor}
\definecolor{red}{rgb}{1,0,0}
\definecolor{col}{rgb}{.85,0,.96}
\usepackage{hyperref}
\usepackage{tikz}
\usetikzlibrary{arrows.meta, positioning}
\hypersetup{
    colorlinks=true,
    linkcolor=blue!70!black,
    citecolor=green!50!black,
    urlcolor=blue!70!black,
    pdfborder={0 0 0}
}
\title{Mean Residual Life Ageing Intensity Function}
\author{
    Ashutosh Singh$^{1}$, Asok K Nanda$^{2}$, Ishapathik Das$^{1}$\textsuperscript{\thanks{Corresponding author \texttt{ishapathik@iittp.ac.in}}}, Sumen Sen$^{3}$, Debasis Sengupta$^{4}$\\
    \small $^1$Indian Institute of Technology Tirupati, India\\
    \small $^2$Indian Institute of Science Education and Research Kolkata, India\\
    \small $^3$Austin Peay State University, Clarksville, TN, USA\\
    \small $^4$Indian Statistical Institute, Kolkata, India
}
\date{} 
\newtheorem{theorem}{Theorem}[section]
\newtheorem{remark}{Remark}[section]
\newtheorem{corollary}{Corollary}[theorem]

\newtheorem{definition}{Definition}[section]
\newtheorem{example}{Counterexample}[section]

\newtheorem{eg}{Example}[section]

\begin{document}
\maketitle
\begin{abstract}
Ageing intensity is usually formulated through the failure rate, whereas its mean residual life (\(MRL\))-based counterpart has not been systematically developed. This paper introduces the mean residual life ageing intensity (\(MRLAI\)) function as an \(MRL\)-based analogue of the failure-rate-based ageing intensity (\(FRAI\)) function. The proposed function compares the current \(MRL\) with its average over \([0,t]\), providing a relative measure of residual-life ageing behaviour. Basic properties of the \(MRLAI\) function are established, including a characterisation of the exponential distribution. The decreasing and increasing \(MRLAI\) classes are introduced and related to the usual \(MRL\)-based ageing classes. Examples are used to clarify the interplay among \(FRAI\), \(MRLAI\), and \(MRL\) monotonicity. Closure and non-closure properties of the proposed classes are examined under mixtures, convolutions, and system formation. An \(MRLAI\)-based stochastic order is also defined, characterised, and compared with classical stochastic orders, and its preservation properties under selected system operations and transformations are investigated.
\end{abstract}
\section{Introduction}

Ageing properties of lifetime distributions are commonly studied through the behaviour of residual lifetimes. Two of the most widely used functions in this context are the hazard rate and the mean residual life (\(MRL\)) function. The hazard rate describes the instantaneous risk of failure, whereas the \(MRL\) function describes the expected remaining lifetime of a unit that has survived up to a specified time. These two functions provide complementary descriptions of ageing and have led to several classical ageing classes and stochastic orders in reliability theory; see, for example, Bryson and Siddiqui~\cite{bryson1969some}, Barlow and Proschan~\cite{barlow1975statistical}, Hollander and Proschan~\cite{hollander1975tests}, Hall and Wellner~\cite{hall1981mean}, Guess and Proschan~\cite{guess1988mean}, Newby~\cite{newby1986applications}, Deshpande et al.~\cite{deshpande1986aspects}, and Nanda et al.~\cite{nanda2010mean}.

Let \(X\) be a non-negative lifetime random variable with distribution function \(F_X\) and survival function \(\bar F_X=1-F_X\). Let \(\tau_X=\sup\{t\geq0:\bar F_X(t)>0\}\) denote the right endpoint of \(X\). The \(MRL\) function of \(X\) is defined by
\begin{equation}\label{eq:MRL}
\mu_X(t)
=
E(X-t\mid X>t)
=
\frac{\displaystyle\int_t^\infty \bar F_X(x)\,dx}
{\bar F_X(t)},
\qquad 0\leq t<\tau_X,
\end{equation}
provided that the integral is finite. The monotonicity of \(\mu_X(t)\) is central in distinguishing decreasing mean residual life (\(DMRL\)) and increasing mean residual life (\(IMRL\)) behaviours. It is also closely connected with other reliability notions, such as new better than used in expectation (\(NBUE\)) and residual-life orderings. In many applications, \(MRL\)-based descriptions are especially useful because the expected remaining lifetime is often more directly interpretable than the instantaneous failure rate.

The ageing intensity (\(AI\)) function was introduced as a relative measure of ageing based on the failure rate. The failure-rate-based \(AI\), denoted by \(FRAI\), compares the instantaneous failure rate with its average value up to time \(t\). Jiang et al.~\cite{jiang2003aging} used this comparison to quantify ageing tendency and to describe lifetime models with non-monotone, including unimodal, failure rates as quasi-increasing, quasi-decreasing, or quasi-constant. Thus, \(FRAI\) provides a way to assess ageing behaviour when the failure rate is not globally monotone.

The \(FRAI\) function has subsequently been studied from several directions. Nanda et al.~\cite{nanda2007properties} investigated its basic properties, the ageing classes induced by its monotonicity, the associated ageing order, and preservation properties under reliability operations. Bhattacharjee et al.~\cite{bhattacharjee2013reliability} further developed its use in reliability analysis and system ageing. Later work extended the \(AI\) framework to distributional characterisation and related ageing measures; see, for example, Szymkowiak~\cite{szymkowiak2018characterizations,szymkowiak2019measures}. Quantile-based and conditionally specified versions were studied by Sunoj and Rasin~\cite{sunoj2018quantile} and Sunoj et al.~\cite{sunoj2020ageing}, respectively. Model-specific investigations for Weibull-type lifetime distributions were considered by Giri et al.~\cite{giri2021ageing}. Generalised reversed \(AI\) functions and multivariate versions based on conditional hazard rates were studied by Buono et al.~\cite{buono2021generalized,buono2022multivariate}.

These contributions show that \(AI\) has been developed mainly through the hazard rate, reversed hazard rate, quantile-based formulations, and conditional hazard-rate models. However, the corresponding construction based directly on the \(MRL\) function has not been systematically investigated. This is a natural gap, since \(MRL\) is a central descriptor of residual lifetime behaviour and is often used alongside the hazard rate in reliability and survival analysis. Moreover, in applications, an estimated \(MRL\) curve obtained from lifetime data may not display strict monotonicity, even when the underlying residual-life behaviour suggests a dominant increasing or decreasing tendency. Motivated by the \(FRAI\) framework, we compare the current value of the \(MRL\) function with its average value over the interval \([0,t]\). This relative comparison indicates whether the present residual-life behaviour is above, below, or close to its historical average, and therefore provides a quantitative way to describe the direction and strength of ageing tendency when strict monotonicity of the estimated \(MRL\) function is not observed. We call the resulting function the mean residual life ageing intensity (\(MRLAI\)) function.

The main contributions of this paper are as follows. First, we introduce the \(MRLAI\) function and study its basic properties. Second, we define increasing and decreasing \(MRLAI\) ageing classes and examine their relationships with the usual \(MRL\)-based ageing classes. Third, we investigate closure and non-closure properties of the proposed ageing classes under standard reliability operations, including mixtures, convolutions, and system formation. Finally, we introduce the \(MRLAI\) order and study its relationship with existing stochastic orders and its preservation properties under selected reliability operations and transformations.

The remainder of the paper is organised as follows. Section~2 introduces the \(MRLAI\) function and studies its basic properties and associated ageing classes. Section~3 investigates closure properties under standard reliability operations. Section~4 introduces the \(MRLAI\) order and studies its main properties. Section~5 concludes the paper.
\section{Mean Residual Life Ageing Intensity Function}

We now formalise the \(MRL\)-based analogue of \(AI\). The proposed measure compares the current \(MRL\) with its average value over \([0,t]\).

\begin{definition}\label{def:MRLAI}
For a non-negative lifetime random variable \(X\) with \(MRL\) function \(\mu_X\), write \(A_X(t)=\int_0^t\mu_X(u)\,du\), \(0<t<\tau_X\). The \(MRLAI\) function is defined by \(L_X^\mu(t)=t\mu_X(t)/A_X(t)\), \(0<t<\tau_X\), whenever \(A_X(t)\) is finite and positive.
\end{definition}

Using \eqref{eq:MRL}, the \(MRLAI\) function can be written as
\[
L_X^{\mu}(t)
=
\frac{
\displaystyle
\frac{t}{\bar F_X(t)}
\int_t^\infty \bar F_X(u)\,du
}{
\displaystyle
\int_0^t
\left\{
\frac{1}{\bar F_X(u)}
\int_u^\infty \bar F_X(z)\,dz
\right\}du
},
\qquad 0<t<\tau_X.
\] Thus, \(L_X^\mu(t)\) measures the current \(MRL\) relative to its average value over \([0,t]\). A larger value of \(L_X^\mu(t)\) indicates a larger present \(MRL\) in relation to historical average $MRL$, corresponding to weaker ageing at time \(t\).

If \(0<E(X)=\mu<\infty\) and \(\mu_X(t)\) is right-continuous at \(t=0\), then
\[
L_X^{\mu}(0)
:=
\lim_{t\downarrow0} L_X^{\mu}(t)
=
1.
\]

The next result shows that the \(MRLAI\) function determines the \(MRL\) function up to its initial value.

\begin{theorem}\label{thm:mrlaif_representation}
Let \(X\) be a non-negative lifetime random variable with
\(\bar F_X(0+)=1\), finite mean, and positive continuous \(MRL\)
function \(\mu_X\) on \((0,\tau_X)\). Then, for \(0<t<\tau_X\),
\[
\mu_X(t)
=
E(X)L_X^\mu(t)
\exp\left\{
\int_0^t\frac{L_X^\mu(u)-1}{u}\,du
\right\},
\]
where the integral is understood as an improper integral.
\end{theorem}

\begin{proof}
Since \(A_X'(t)=\mu_X(t)\),
\[
L_X^\mu(t)
=
\frac{tA_X'(t)}{A_X(t)},
\qquad 0<t<\tau_X.
\]
Consequently,
\[
\frac{d}{dt}
\log\left(\frac{A_X(t)}{t}\right)
=
\frac{L_X^\mu(t)-1}{t}.
\]
Integrating over \((\varepsilon,t)\), where
\(0<\varepsilon<t<\tau_X\), gives
\[
\log\left(\frac{A_X(t)}{t}\right)
-
\log\left(\frac{A_X(\varepsilon)}{\varepsilon}\right)
=
\int_\varepsilon^t
\frac{L_X^\mu(u)-1}{u}\,du.
\]
Since \(\bar F_X(0+)=1\) and \(E(X)<\infty\),
\(\mu_X(t)\to E(X)\) as \(t\downarrow0\), and hence
\[
\frac{A_X(\varepsilon)}{\varepsilon}
\longrightarrow E(X)
\qquad\text{as }\varepsilon\downarrow0.
\]
Therefore,
\[
A_X(t)
=
E(X)t
\exp\left\{
\int_0^t\frac{L_X^\mu(u)-1}{u}\,du
\right\}.
\]
Since \(A_X'(t)=L_X^\mu(t)A_X(t)/t\), the required representation
follows.
\end{proof}
The following result gives an alternative representation of the reliability function by combining the \(MRL\)-based survival representation with Theorem~\ref{thm:mrlaif_representation}.

\begin{corollary}\label{thm:survival_mrlai_strength}
Let \(X\) be a non-negative lifetime random variable with \(\bar F_X(0+)=1\), finite mean \(\mu=E(X)\), and positive continuous \(MRL\) function \(\mu_X\) on \((0,\tau_X)\). For \(0<t<\tau_X\), define
\[
S_X(t)
=
\log L_X^\mu(t)
+
\int_0^t\frac{L_X^\mu(u)-1}{u}\,du,
\]
where the integral at the origin is taken in the improper sense. Then
\[
\bar F_X(t)
=
e^{-S_X(t)}
\exp\left\{
-\frac{1}{\mu}\int_0^t e^{-S_X(u)}\,du
\right\},
\qquad 0<t<\tau_X.
\]
\end{corollary}

\begin{proof}
By Theorem~\ref{thm:mrlaif_representation},
\[
\mu_X(t)
=
\mu L_X^\mu(t)
\exp\left\{
\int_0^t\frac{L_X^\mu(u)-1}{u}\,du
\right\}
=
\mu e^{S_X(t)},
\qquad 0<t<\tau_X.
\]
Consequently,
\[
\frac{1}{\mu_X(t)}
=
\frac{1}{\mu}e^{-S_X(t)}.
\]
The survival function corresponding to the \(MRL\) function satisfies
\begin{equation}\label{eq:mrl-survival-representation}
\bar F_X(t)
=
\frac{\mu_X(0+)}{\mu_X(t)}
\exp\left\{
-\int_0^t\frac{du}{\mu_X(u)}
\right\},
\qquad 0<t<\tau_X.
\end{equation}
Since \(\mu_X(0+)=E(X)=\mu\), we have
\[
\frac{\mu_X(0+)}{\mu_X(t)}
=
e^{-S_X(t)}
\]
and
\[
\int_0^t\frac{du}{\mu_X(u)}
=
\frac{1}{\mu}\int_0^t e^{-S_X(u)}\,du.
\]
Substitution into \eqref{eq:mrl-survival-representation} yields
\[
\bar F_X(t)
=
e^{-S_X(t)}
\exp\left\{
-\frac{1}{\mu}\int_0^t e^{-S_X(u)}\,du
\right\},
\qquad 0<t<\tau_X.
\]
This completes the proof.
\end{proof}
For two lifetime random variables \(X\) and \(Y\), let
\(\tau_{XY}=\min\{\tau_X,\tau_Y\}\).

\begin{theorem}\label{cor:mrlaif_proportionality}
The equality \(L_X^\mu(t)=L_Y^\mu(t)\) holds for all \(0<t<\tau_{XY}\) if and only if \(\mu_X(t)=c\,\mu_Y(t)\) for all \(0<t<\tau_{XY}\), for some constant \(c>0\).
\end{theorem}
\begin{proof}
Suppose \(L_X^\mu(t)=L_Y^\mu(t)\) for all \(0<t<\tau_{XY}\). Then
$$\frac{d}{dt}\left[\log\left(\frac{A_X(t)}{A_Y(t)}\right)\right]=0,\qquad 0<t<\tau_{XY},$$ 
so that \(A_X(t)=c\,A_Y(t)\) for some \(c>0\) and all \(0<t<\tau_{XY}\), which, when differentiated, gives \(\mu_X(t)=c\mu_Y(t)\) for all \(0<t<\tau_{XY}\). The proof of the converse part is easy and hence omitted. 
\end{proof}

It is also of interest to identify which positive functions can arise as \(MRLAI\) functions. The following theorem gives an admissibility criterion for a proposed \(MRLAI\) function.

Fix \(\tau_X\in(0,\infty]\), and let \(\ell\) be a positive continuously differentiable function on \((0,\tau_X)\) such that \(\ell(t)\to1\) as \(t\downarrow0\) and \(H_\ell(t)=\int_0^t\{\ell(u)-1\}/u\,du\) is finite for every \(0<t<\tau_X\). Define
\[
g_\ell(t)=\ell(t)\exp\{H_\ell(t)\},\qquad
c_\ell^*=\inf\left\{-\frac{1}{g_\ell'(t)}:g_\ell'(t)<0,\ 0<t<\tau_X\right\},
\]
with the convention \(c_\ell^*=+\infty\) if \(g_\ell'(t)\geq0\) throughout \((0,\tau_X)\). If \(g_\ell'\) assumes both positive and negative values, the infimum
is taken only over the set on which \(g_\ell'(t)<0\).

\begin{theorem}\label{thm:admissible_mrlai_iff}
Under the above assumptions, \(\ell\) is the \(MRLAI\) function of a lifetime distribution \(X\) with right endpoint of the support \(\tau_X\), \(\bar F_X(0+)=1\), and finite mean if and only if \(c_\ell^*>0\) and
\[
\lim_{t\uparrow\tau_X}\int_0^t\frac{du}{g_\ell(u)}=+\infty.
\]
Moreover, for every finite \(c>0\) satisfying \(c\leq c_\ell^*\), an associated \(MRL\) function is \(\mu_X(t)=cg_\ell(t)\), \(0<t<\tau_X\).
\end{theorem}

\begin{proof}
Suppose first that \(\ell\) is the \(MRLAI\) function of a lifetime random variable \(X\) with \(MRL\) function \(\mu_X\), and let \(c=\mu_X(0+)=E(X)\). By Theorem~\ref{thm:mrlaif_representation},
\[
\mu_X(t)=c\ell(t)\exp\left\{\int_0^t\frac{\ell(u)-1}{u}\,du\right\}
=cg_\ell(t),\qquad 0<t<\tau_X.
\]
 Since a continuously differentiable \(MRL\) function satisfies \(1+\mu_X'(t)\geq0\) on \((0,\tau_X)\) (cf.\ \cite[p.~82]{shaked2007stochastic}), we have \(1+cg_\ell'(t)\geq0\). Hence, whenever \(g_\ell'(t)<0\), \(c\leq-1/g_\ell'(t)\), and therefore \(c\leq c_\ell^*\). At points where \(g_\ell'(t)\geq0\), the admissibility condition is automatic. Consequently, the conclusion remains valid when \(g_\ell'\) changes sign. Since \(c>0\), \(c_\ell^*>0\); if \(g_\ell'\geq0\) on \((0,\tau_X)\), then \(c_\ell^*=+\infty\).

By equation~\eqref{eq:mrl-survival-representation},
\[
\mu_X(t)\bar F_X(t)
=c\exp\left\{-\int_0^t\frac{du}{\mu_X(u)}\right\}.
\]
Since \(\mu_X(t)\bar F_X(t)=\int_t^\infty\bar F_X(u)\,du\to0\) as \(t\uparrow\tau_X\), it follows that \(\int_0^tdu/\mu_X(u)\to+\infty\) as $t\uparrow\tau_X$. Using \(\mu_X(t)=cg_\ell(t)\), we obtain
\[
\lim_{t\uparrow\tau_X}\int_0^t\frac{du}{g_\ell(u)}
=c\lim_{t\uparrow\tau_X}\int_0^t\frac{du}{\mu_X(u)}
=+\infty,
\]
which proves the necessity.

Conversely, suppose that \(c_\ell^*>0\) and \(\int_0^tdu/g_\ell(u)\to+\infty\) as \(t\uparrow\tau_X\). For any finite \(c\) satisfying \(0<c\leq c_\ell^*\), set \(\mu_X(t)=cg_\ell(t)\), \(0<t<\tau_X\). Since \(H_\ell'(t)=\{\ell(t)-1\}/t\), we have
\[
\frac{d}{dt}\left[ct\exp\{H_\ell(t)\}\right]
=c\ell(t)\exp\{H_\ell(t)\}
=\mu_X(t).
\]
Since \(H_\ell(t)\to0\) and \(ct\exp\{H_\ell(t)\}\to0\) as \(t\downarrow0\), it follows that \(\int_0^t\mu_X(u)\,du=ct\exp\{H_\ell(t)\}\). Hence,
\[
\frac{t\mu_X(t)}{\displaystyle\int_0^t\mu_X(u)\,du}
=\frac{tcg_\ell(t)}{ct\exp\{H_\ell(t)\}}
=\ell(t),\qquad 0<t<\tau_X.
\]

If \(g_\ell'(t)\geq0\), then \(1+cg_\ell'(t)\geq1\); if \(g_\ell'(t)<0\), then \(c\leq c_\ell^*\leq-1/g_\ell'(t)\), and therefore \(1+cg_\ell'(t)\geq0\). Thus, \(1+cg_\ell'(t)\geq0\) throughout \((0,\tau_X)\). Define
\[
\bar F_X(t)=\frac{1}{g_\ell(t)}
\exp\left\{-\frac{1}{c}\int_0^t\frac{du}{g_\ell(u)}\right\},
\qquad 0<t<\tau_X,
\]
and set \(\bar F_X(0)=1\); if \(\tau_X<\infty\), set \(\bar F_X(t)=0\) for \(t\geq\tau_X\). Since \(\ell(t)\to1\) and \(H_\ell(t)\to 0\) as \(t\downarrow 0\), we have
\(g_\ell(t)\to1\) as \(t\downarrow 0\), and hence \(\bar F_X(0+)=1\). Moreover,
\[
-\frac{\bar F_X'(t)}{\bar F_X(t)}
=\frac{1+cg_\ell'(t)}{cg_\ell(t)}
\geq0,\qquad 0<t<\tau_X,
\]
so that \(\bar F_X\) is non-increasing.

Furthermore,
\[
\mu_X(t)\bar F_X(t)
=c\exp\left\{-\frac{1}{c}\int_0^t\frac{du}{g_\ell(u)}\right\}\]
and
\[\{\mu_X(t)\bar F_X(t)\}'=-\bar F_X(t).
\]
Thus, we have \(\mu_X(t)\bar F_X(t)\to 0\) as \(t\uparrow\tau_X\). Integrating the preceding identity over \([t,b]\) and letting \(b\uparrow\tau_X\) give
\[
\int_t^\infty\bar F_X(u)\,du
=\mu_X(t)\bar F_X(t),\qquad 0<t<\tau_X.
\]
Thus, letting \(t\downarrow0\), we obtain \(\int_0^\infty\bar F_X(u)\,du=c<\infty\). Thus, \(\bar F_X\) is a proper survival function with right endpoint of the support \(\tau_X\), \(\bar F_X(0+)=1\), and finite mean \(c\), possibly with an atom at \(\tau_X\) when \(\tau_X<\infty\). Finally,
\[
\mu_X(t)=\frac{\displaystyle\int_t^\infty\bar F_X(u)\,du}{\bar F_X(t)},
\qquad 0<t<\tau_X,
\]
and therefore \(\mu_X\) is the \(MRL\) function associated with \(\bar F_X\), with corresponding \(MRLAI\) function \(\ell\).
\end{proof}

\subsection{Basic properties of the \(MRLAI\) function}
We now study some basic properties of the \(MRLAI\) function. The first result gives a characterisation of the exponential distribution through the constancy of \(MRLAI\).

Suppose \(m=\mathbb{E}(X)\). Since \(X\) is non-negative and has finite mean, the tail-integral formula gives
\[
m=\int_0^\infty \bar F_X(u)\,du.
\]
Clearly, \(A_X(t)=\int_0^t\mu_X(u)\,du\) is finite for every \(0<t<\tau_X\). Furthermore,
\[
\frac{A_X(t)}{t}=\frac{1}{t}\int_0^t\mu_X(u)\,du\longrightarrow m
\qquad\text{as }t\downarrow0.
\]
Suppose that \(L_X^\mu(t)=c(>0)\) for all \(0<t<\tau_X\). Since \(\mu_X\) is positive and continuous, \(A_X(t)>0\) and \(A_X'(t)=\mu_X(t)\) on \((0,\tau_X)\). Hence
\[
\frac{A_X'(t)}{A_X(t)}=\frac{c}{t},
\qquad 0<t<\tau_X.
\]
Integration gives \(A_X(t)=Kt^c\) for some \(K>0\). Therefore, \(A_X(t)/t=Kt^{c-1}\). Since this expression converges to \(m>0\) as \(t\downarrow0\), we have \(c=1\) and \(K=m\). Consequently,
\[
A_X(t)=mt\quad\text{and}\quad\mu_X(t)=m,
\qquad 0<t<\tau_X.
\]

We next show that the distribution cannot have a finite right endpoint. Suppose that \(\tau_X<\infty\). Since \(\bar F_X\) is nonincreasing and \(\bar F_X(u)=0\) for \(u>\tau_X\),
\[
m=\mu_X(t)=\frac{\displaystyle\int_t^{\tau_X}\bar F_X(u)\,du}{\bar F_X(t)}
\leq\frac{(\tau_X-t)\bar F_X(t)}{\bar F_X(t)}
=\tau_X-t.
\]
Hence,
\[
0<m\leq\inf_{0<t<\tau_X}(\tau_X-t)=0,
\]
which is a contradiction. Therefore, the support of the distribution can't be finite, so \(\mu_X(t)=m\), for \(t\in(0,\infty)\). The defining relation of the \(MRL\) function now gives
\[
m\bar F_X(t)=\int_t^\infty\bar F_X(u)\,du
=m-\int_0^t\bar F_X(u)\,du,
\qquad t>0.
\]
Equivalently,
\[
\bar F_X(t)=1-\frac{1}{m}\int_0^t\bar F_X(u)\,du,
\qquad t>0,
\]
so that,
\[
\bar F_X'(t)=-\frac{1}{m}\bar F_X(t),
\qquad t>0.
\]
Consequently, we have
\[
\bar F_X(t)=\exp\left(-\frac{t}{m}\right),
\qquad t>0.
\]
Hence \(X\) is exponentially distributed with mean \(m\). Conversely, if \(X\) is exponentially distributed with mean \(m\), then \(\mu_X(t)=m\) and \(A_X(t)=mt\). Therefore,
\[
L_X^\mu(t)=\frac{t\mu_X(t)}{A_X(t)}=1,
\qquad t>0.
\]

Thus, we have the following theorem:
\begin{theorem}\label{thm:exp_characterization}
Let \(X\) be a continuous lifetime random variable having \(MRLAI\) function \(L_X^\mu\). If \(L_X^\mu(t)\) is independent of \(t\) for all \(t\) over the support \(X\), then the support cannot be finite. Further, if \(L_X^\mu(t)=c\) for all \(t\in(0,\infty)\), then the distribution of \(X\) is exponential and conversely.
\end{theorem}
\begin{remark}
For the failure-rate-based \(AI\), Nanda et al.~\cite{nanda2007properties} showed that \(L_X(t)=c\), \(c>0\), characterises the Weibull distribution with shape parameter \(c\). In contrast, Theorem~\ref{thm:exp_characterization} shows that constancy of \(MRLAI\) characterises the exponential distribution and, in this case, \(L_X^\mu(t)\equiv 1\).
\end{remark}
\hspace*{.2in} The next two theorems relate the monotonicity of the \(MRL\) function to the position of the \(MRLAI\) function relative to one. 
\begin{theorem}\label{thm:2.5}
Let $X$ be a random variable defined on $[0,\tau_X]$.
If \(\mu_X\) is nonincreasing on \([0,\tau_X)\), then \(L_X^{\mu}(t)\leq 1\) for all \(0<t<\tau_X\). The converse is not necessarily true.
\end{theorem}
\begin{proof}
    If \(\mu_X\) is nonincreasing, then \(\mu_X(u)\geq \mu_X(t)\) for \(0<u<t\). Hence \(t^{-1}\int_0^t \mu_X(u)\,du \geq \mu_X(t)\), which gives \(L_X^\mu(t)\leq1\) for all \(0<t<\tau_X\). For the converse part, we consider the $MRL$ function as  
 \[
\mu_X(t)=
\begin{cases}
1-0.4t, & 0\le t<2,\\
0.1t, & 2\le t<3,\\
0.3, & t\ge3.
\end{cases}
\]
  so that the $MRLAI$ function becomes
  \[
L_X^\mu(t)=
\begin{cases}
\dfrac{1-0.4t}{1-0.2t}, & 0<t<2,\\[6pt]
\dfrac{0.1t^2}{1+0.05t^2}, & 2\le t<3,\\[6pt]
\dfrac{0.3t}{0.3t+0.55}, & t\ge3,
\end{cases}
\]
which is non-monotone with $L_{X}^{\mu}(t) < 1$, for all $t > 0$, as Figure~\ref{mrlai}(a) shows. Also, note that $L_X^{\mu}(t)$ increases in $t \geq 3$ with $\lim_{t\to\infty}L_X^\mu(t)=1.$ 
  \end{proof}
\begin{theorem}\label{thm:2.4}
Let $X$ be a random variable defined on $[0,\tau_X]$.
If \(\mu_X\) is nondecreasing on \([0,\tau_X)\), then \(L_X^\mu(t)\geq1\) for all \(0<t<\tau_X\). The converse need not hold.
\end{theorem}
\begin{proof}
   If \(\mu_X\) is nondecreasing, then \(\mu_X(u)\leq \mu_X(t)\) for \(0<u<t\). Hence \(t^{-1}\int_0^t \mu_X(u)\,du \leq \mu_X(t)\), which gives \(L_X^\mu(t)\geq1\) for all \(0<t<\tau_X\). That the converse is not true is shown by considering the $MRL$ function as
  \[
\mu_X(t)=
\begin{cases}
1+t, & 0\le t<5,\\
11-t, & 5\le t<6,\\
5, & t\ge6.
\end{cases}
\]
  This gives the $MRLAI$ function as
\[
L_X^\mu(t)=
\begin{cases}
\dfrac{2(t+1)}{t+2}, & 0<t<5,\\[6pt]
\dfrac{2t(11-t)}{-t^2+22t-50}, & 5\le t<6,\\[6pt]
\dfrac{5t}{5t-7}, & t\ge 6,
\end{cases}
\]
which is non-monotone with $L_X^{\mu}(t) > 1$, for all $t > 0,$ as Figure~\ref{mrlai}(b) shows. Note that $ L_X^{\mu}(t)$ is decreasing in $t \geq 6$ with $\lim\limits_{t \to \infty} L_X^{\mu}(t) = 1.$
\end{proof}
\begin{figure}[ht]
	\centering
	\includegraphics[width=0.75\textwidth]{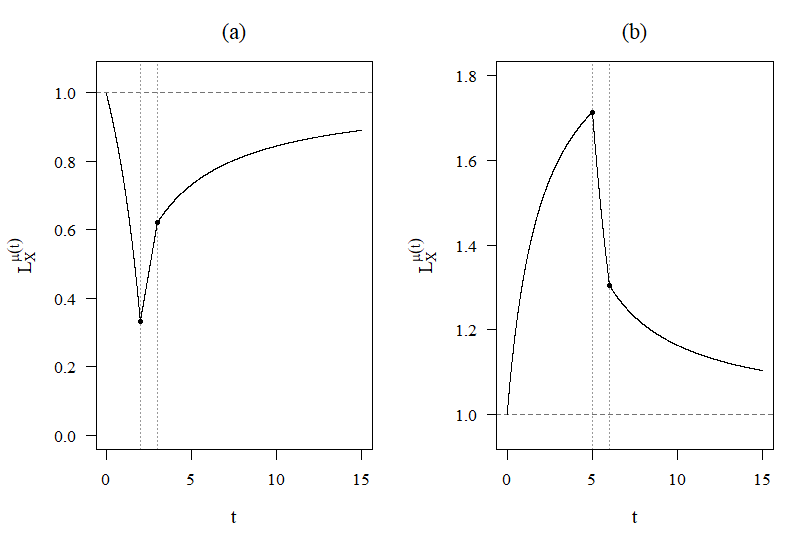}
	\caption{(a) The $MRLAI$ function given in Theorem~\ref{thm:2.5}; (b) the $MRLAI$ function given in Theorem~\ref{thm:2.4}.}
	\label{mrlai}
\end{figure}
The next result provides an asymptotic interpretation of the \(MRLAI\) function. For the cumulative \(MRL\) function \(A_X\), the quantity \(d\log A_X(t)/d\log t=tA_X'(t)/A_X(t)\) represents the point elasticity of \(A_X(t)\) with respect to \(t\), that is, the local proportional change in \(A_X(t)\) resulting from a proportional change in \(t\); see \cite{chen2017elasticity}. Since \(A_X'(t)=\mu_X(t)\), the \(MRLAI\) function \(L_X^\mu(t)\) is precisely the point elasticity, or logarithmic growth rate, of \(A_X(t)\). Further, if \(X\) has unbounded support and \(t\mu_X'(t)/\mu_X(t)\to\rho>-1\) as \(t\to\infty\), then \(\mu_X(t)\) is regularly varying with index \(\rho\), and Karamata's integration theorem yields \(A_X(t)\sim t\mu_X(t)/(1+\rho)\) as \(t\to\infty\); see \cite{bingham1987regular}. Consequently, \(A_X(t)\) is regularly varying with index \(1+\rho\) and \(L_X^\mu(t)\to1+\rho\) as \(t\to\infty\). Thus, the limiting value of the \(MRLAI\) function gives the growth index of \(A_X(t)\), while one less than this limit gives the corresponding growth index of \(\mu_X(t)\).

\begin{theorem}\label{thm:asymptotic_mrlai}
Let \(X\) have finite mean and a positive continuous \(MRL\) function \(\mu_X\) on the support \((0,\tau_X)\). Then \(L_X^\mu(t)=d\log A_X(t)/d\log t\), for all \(0<t<\tau_X\). Moreover, if \(X\) has unbounded support, \(\mu_X\) is differentiable, and \(t\mu_X'(t)/\mu_X(t)\to\rho>-1\) as \(t\to\infty\), then \(L_X^\mu(t)\to1+\rho\) as \(t\to\infty\).
\end{theorem}

\begin{proof}
Since \(A_X'(t)=\mu_X(t)\), we have
\[
\frac{d\log A_X(t)}{d\log t}
=
\frac{tA_X'(t)}{A_X(t)}
=
\frac{t\mu_X(t)}{A_X(t)}
=
L_X^\mu(t),
\qquad 0<t<\tau_X.
\]
For the second assertion, suppose that \(X\) has unbounded support and \(t\mu_X'(t)/\mu_X(t)\to\rho>-1\) as \(t\to\infty\). Choose \(\varepsilon>0\) such that \(\rho-\varepsilon>-1\). Then, for some \(T>0\),
\[
\frac{t\mu_X'(t)}{\mu_X(t)}
>
\rho-\varepsilon,
\text{ for all } t\geq T.
\]
Integrating over \([T,t]\) gives, for all $t\geq T$, 
\[
\mu_X(t)
\geq
\mu_X(T)\left(\frac{t}{T}\right)^{\rho-\varepsilon}.
\]
Since \(\rho-\varepsilon>-1\), both \(A_X(t)\) and \(t\mu_X(t)\) tend to infinity as \(t\to\infty\). Moreover,
\[
1+\frac{t\mu_X'(t)}{\mu_X(t)}
\longrightarrow
1+\rho>0,
\]
so that \(\mu_X(t)+t\mu_X'(t)>0\) for all sufficiently large \(t\). Therefore, by L'H\^{o}pital's rule,
\[
\lim_{t\to\infty}\frac{A_X(t)}{t\mu_X(t)}
=
\lim_{t\to\infty}
\frac{\mu_X(t)}
{\mu_X(t)+t\mu_X'(t)}
=
\frac{1}{1+\rho}.
\]
Consequently,
\[
\lim_{t\to\infty}L_X^\mu(t)
=
\lim_{t\to\infty}\frac{t\mu_X(t)}{A_X(t)}
=
1+\rho.
\]
\end{proof}
Below we take an example to show that the conditions of the above theorem are met by some random variables.
\begin{eg}
Let \(X\) have the Pareto type II survival function
\[
\bar F_X(t)=\left(1+\frac{t}{\beta}\right)^{-\alpha},
\qquad t\geq0,
\]
where \(\alpha>1\) and \(\beta>0\). Then \(X\) has unbounded support, \(E(X)=\beta/(\alpha-1)<\infty\), and
\[
\mu_X(t)=\frac{\beta+t}{\alpha-1},
\qquad
\frac{t\mu_X'(t)}{\mu_X(t)}
=\frac{t}{\beta+t}\longrightarrow1
\quad\text{as }t\to\infty.
\]
Thus, the assumptions of Theorem~\ref{thm:asymptotic_mrlai} are satisfied with \(\rho=1\). Moreover,
\[
A_X(t)=\frac{\beta t+t^2/2}{\alpha-1},
\qquad
L_X^\mu(t)
=\frac{t\mu_X(t)}{A_X(t)}
=\frac{2(\beta+t)}{2\beta+t}
\longrightarrow2=1+\rho
\quad\text{as }t\to\infty.
\]
\hfill\(\Box\)
\end{eg}
  \subsection{Non-parametric classes based on the MRLAI function}

\hspace*{.2in} We next introduce two non-parametric ageing classes based on the monotonicity of the \(MRLAI\) function.

\begin{definition}
A lifetime random variable \(X\) is said to be decreasing \(MRLAI\)
(\(DMRLAI\)) if \(L_X^\mu(t)\) is nonincreasing on
\((0,\tau_X)\). It is said to be increasing \(MRLAI\) (\(IMRLAI\))
if \(L_X^\mu(t)\) is nondecreasing on \((0,\tau_X)\).
\end{definition}

A nonincreasing \(MRLAI\) function indicates that the current \(MRL\), relative to its average over \([0,t]\), decreases with \(t\). This provides a natural notion of ageing.\\

Since \(L_X^\mu(t)\) compares the current \(MRL\) with its average over \([0,t]\), it is natural to interpret the inequalities \(L_X^\mu(t)\leq1\) and \(L_X^\mu(t)\geq1\) in terms of the average \(MRL\) function.

For \(0<t<\tau_X\), let \(\bar\mu_X(t)=A_X(t)/t\). The random variable \(X\) is said to be decreasing in mean residual life average (\(DMRLA\)) if \(\bar\mu_X\) is nonincreasing on \((0,\tau_X)\), and increasing in mean residual life average (\(IMRLA\)) if \(\bar\mu_X\) is nondecreasing on \((0,\tau_X)\).

\begin{theorem}\label{thm:mrlai_average_mrl}

Let \(X\) have a positive continuous \(MRL\) function on the support \([0,\tau_X)\). Then \(X\in DMRLA\) if and only if \(L_X^\mu(t)\leq1\) for all \(0<t<\tau_X\), and \(X\in IMRLA\) if and only if \(L_X^\mu(t)\geq1\) for all \(0<t<\tau_X\).
 \end{theorem}

\begin{proof}
Since \(\bar\mu_X(t)=A_X(t)/t\), differentiation gives
\[
\bar\mu_X'(t)=\frac{t\mu_X(t)-A_X(t)}{t^2}
=\frac{A_X(t)}{t^2}\{L_X^\mu(t)-1\}\stackrel{sign}{=}L_X^\mu(t)-1.
\]
Hence \(\bar\mu_X(t)\) is nonincreasing if and only if \(L_X^\mu(t)\leq1\), and \(\bar\mu_X(t)\) is nondecreasing if and only if \(L_X^\mu(t)\geq1\). This proves the result.
\end{proof}


For a positive function \(A\) on \((0,\tau_X)\), call \(A\) geometrically concave
(resp. geometrically convex) if \(y\mapsto \log A(e^y)\) is concave
(resp. convex) on its logarithmic domain. The next result gives a structural characterisation of the two classes.

\begin{theorem}\label{thm:geometric_characterization}
Let \(X\) have a positive continuous \(MRL\) function \(\mu_X\) on the support
\([0,\tau_X)\). Then \(X\) is \(DMRLAI\) if and only if \(A_X\) is geometrically concave. Similarly, \(X\) is \(IMRLAI\) if and only if \(A_X\) is geometrically convex.
\end{theorem}

\begin{proof}
Write \(\phi(y)=\log A_X(e^y)\). Since \(A_X'(t)=\mu_X(t)\), we have
\[
\phi'(y)
=
\frac{e^yA_X'(e^y)}{A_X(e^y)}
=
\frac{e^y\mu_X(e^y)}{\int_0^{e^y}\mu_X(u)\,du}
=
L_X^\mu(e^y).
\]
Since \(e^y\) is increasing, \(L_X^\mu(t)\) is nonincreasing in \(t\) if and only if \(\phi'(y)\) is nonincreasing in \(y\). This is equivalent to the concavity of \(\phi\), and hence to the geometric concavity of \(A_X\). The \(IMRLAI\) case follows similarly, with convexity replacing concavity.
\end{proof}
The preceding results describe the \(MRLAI\) function of the original lifetime variable.
Since the \(MRL\) function is intrinsically connected with residual lifetimes, it is natural
to ask how the \(MRLAI\) function behaves after conditioning on survival up to a fixed age.
The following result gives the corresponding expression.

For a lifetime \(X\) with right endpoint \(\tau_X\) and positive
continuous \(MRL\) function \(\mu_X\), let
\(X_t=(X-t\mid X>t)\) denote its residual lifetime at age
\(0<t<\tau_X\). For \(0<s<\tau_X-t\),
\[
\bar F_{X_t}(s)
=\frac{\bar F_X(t+s)}{\bar F_X(t)},
\qquad
\mu_{X_t}(s)
=\frac{\displaystyle\int_{t+s}^{\tau_X}\bar F_X(u)\,du}
       {\bar F_X(t+s)}
=\mu_X(t+s).
\]
Hence,
\[
L_{X_t}^{\mu}(s)
=\frac{s\mu_X(t+s)}
       {\displaystyle\int_t^{t+s}\mu_X(u)\,du}
=\frac{sA_X'(t+s)}
       {A_X(t+s)-A_X(t)}.
\]
Moreover,
\[
A_{X_t}(s)
=\int_0^s\mu_X(t+u)\,du
=A_X(t+s)-A_X(t).
\]
Therefore, by Theorem~\ref{thm:geometric_characterization},
\(X_t\) is \(DMRLAI\) if and only if
\(s\mapsto A_X(t+s)-A_X(t)\) is geometrically concave, and is
\(IMRLAI\) if and only if this function is geometrically convex.

The following theorem gives a basic implication for the \(DMRLAI\) and \(IMRLAI\) classes.

\begin{theorem}\label{2.6}
Let \(X\) have finite positive mean and a positive continuous \(MRL\) function on \([0,\tau_X)\). If \(X\) is \(DMRLAI\) \((resp.\;IMRLAI)\), then \(L_X^\mu(t)\leq (resp.\;\geq)\;1\) for all \(0<t<\tau_X\), where \(\tau_X\) is the upper endpoint of the support of \(X\).
\end{theorem}

\begin{proof}
Let \(X\) be a \(DMRLAI\) random variable. Then \(L_X^\mu(t)\) is nonincreasing for \(t\in(0,\tau_X)\). Hence,
\[
\sup_{t\in(0,\tau_X)}L_X^\mu(t)
=
\lim_{t\downarrow0}L_X^\mu(t)
=
1.
\]
Thus \(L_X^\mu(t)\leq1\) for all \(0<t<\tau_X\). The proof of the parenthetical statement is similar.
\end{proof}

Combining Theorems~\ref{thm:mrlai_average_mrl} and~\ref{2.6}, we obtain
\(DMRLAI\subsetneq DMRLA\) and \(IMRLAI\subsetneq IMRLA\). The
converses do not hold in general. Indeed, the counterexample in
Theorem~\ref{thm:2.5}, illustrated in Figure~\ref{mrlai}(a), satisfies
\(L_X^\mu(t)<1\) for all \(t>0\), while \(L_X^\mu(t)\) is
non-monotone; hence \(X\in DMRLA\setminus DMRLAI\). Similarly, the
counterexample in Theorem~\ref{thm:2.4}, illustrated in
Figure~\ref{mrlai}(b), satisfies \(L_X^\mu(t)>1\) for all \(t>0\),
while \(L_X^\mu(t)\) is non-monotone; hence
\(X\in IMRLA\setminus IMRLAI\).


\begin{remark}
If \(X\) is both \(IMRLAI\) and \(DMRLAI\), then \(L_X^\mu(t)\) is constant. Hence, by Theorem~\ref{thm:exp_characterization}, the underlying distribution is exponential.\hfill$\Box$
\end{remark}

\hspace*{.2in} The following theorem characterises the \(DMRLAI\) and \(IMRLAI\) classes in terms of the \(MRL\) function.

\begin{theorem}\label{2.7}
Assume that \(\mu_X\) is differentiable and positive on the support \((0,\tau_X)\). Then \(X\) is \(DMRLAI\) \((resp.\;IMRLAI)\) if and only if
\[
L_X^\mu(t)\geq(resp.\;\leq)\;1+\frac{t\mu_X'(t)}{\mu_X(t)},
\qquad 0<t<\tau_X.
\]
\end{theorem}

\begin{proof}
We prove the \(DMRLAI\) case. Since \(X\) is \(DMRLAI\) if and only if \(L_X^\mu(t)\) is nonincreasing, $i.e.$,
\[
\frac{d}{dt}
\left(
\frac{t\mu_X(t)}{\int_0^t\mu_X(u)\,du}
\right)
\leq 0.
\]
This is equivalent to
\[
\frac{t\mu_X'(t)+\mu_X(t)}{\mu_X(t)}
\leq
\frac{t\mu_X(t)}{\int_0^t\mu_X(u)\,du},
\]
and hence
\[
L_X^\mu(t)\geq 1+\frac{t\mu_X'(t)}{\mu_X(t)}.
\]
The \(IMRLAI\) case is analogous.
\end{proof}

\begin{remark}\label{2.1}
Combining Theorems~\ref{2.6} and~\ref{2.7}, we have the following.
\begin{enumerate}
\item[(a)] If \(X\in DMRLAI\), then, for all \(0<t<\tau_X\),
\[
1+\frac{t\mu_X'(t)}{\mu_X(t)}
\leq
L_X^\mu(t)
\leq
1,
\]
which implies that \(X\) is \(DMRL\). To see that the converse may not hold, consider an Erlang random variable \(X\) with density
\[
f_X(t)=\frac{1}{2}t^2e^{-t}, \qquad t>0.
\]
Its \(MRL\) function is given by
\[
\mu_X(t)
=
\frac{t^2+4t+6}{t^2+2t+2},
\]
which is decreasing. The corresponding \(MRLAI\) function is obtained as
\[
L_X^\mu(t)
=
\frac{t(t^2+4t+6)}
{(t^2+2t+2)
\left[
\ln(t^2+2t+2)+2\tan^{-1}(t+1)+t-\ln 2-\frac{\pi}{2}
\right]},
\]
which is non-monotone, since
\[
L_X^\mu(2.5)\approx0.7767024,\qquad
L_X^\mu(5.0)\approx0.7525321,\qquad
L_X^\mu(10)\approx0.7720608.
\]
Thus, a \(DMRL\) distribution can have a non-monotone \(MRLAI\) function. Hence, the \(DMRLAI\) class is a proper subset of the \(DMRL\) class.

\item[(b)] If \(X\in IMRLAI\), then, for all \(0<t<\tau_X\),
\[
1
\leq
L_X^\mu(t)
\leq
1+\frac{t\mu_X'(t)}{\mu_X(t)}.
\]
This implies \(\mu_X'(t)\geq0\), and hence \(X\) is \(IMRL\). To see that the converse may not hold, consider a random variable \(X\) with survival function given by
\[
\bar F_X(t)
=
\frac{(t+3\sqrt{t}+2)e^{-t-2\sqrt{t}}}{2},
\qquad t\geq0.
\]
The corresponding \(MRL\) function is obtained as
\[
\mu_X(t)
=
\frac{\sqrt{t}+1}{\sqrt{t}+2},
\]
which is increasing on \([0,\infty)\). However, the corresponding \(MRLAI\) function is obtained as
\[
L_X^\mu(t)
=
\frac{(\sqrt{t}+1)t}
{(\sqrt{t}+2)\left(t-2\sqrt{t}+4\ln(\sqrt{t}+2)-4\ln 2\right)},
\]
which is non-monotone, since
\[
L_X^\mu(0.01)\approx1.015006,\qquad
L_X^\mu(5.50)\approx1.082015,\qquad
L_X^\mu(15.00)\approx1.076368.
\]
Thus, the \(IMRLAI\) class is a proper subset of the \(IMRL\) class.\hfill$\Box$
\end{enumerate}
\end{remark}

The following example illustrates a case where both the \(MRL\) and the \(MRLAI\) functions are non-monotone.

Consider a lifetime random variable \(X\) with the piecewise linear \(MRL\) function as
\[
\mu_X(t)=
\begin{cases}
4-t, & 0\leq t\leq2,\\[4pt]
2+0.5(t-2), & 2<t\leq4,\\[4pt]
3-0.5(t-4), & 4<t\leq6,\\[4pt]
2, & t>6.
\end{cases}
\]
This is a valid \(MRL\) function and is non-monotone. The corresponding \(MRLAI\) function comes out to be
\[
L_X^\mu(t)=
\begin{cases}
\dfrac{2(4-t)}{8-t}, & 0<t\leq2,\\[8pt]
\dfrac{4t(1+0.5t)}{12+4t+t^2}, & 2<t\leq4,\\[8pt]
\dfrac{t\bigl(3-0.5(t-4)\bigr)}{-5+5t-\tfrac{t^2}{4}}, & 4<t\leq6,\\[8pt]
\dfrac{t}{t+2}, & t>6.
\end{cases}
\]
It can be verified that
\[
L_X^\mu(2)\approx0.667,\qquad
L_X^\mu(4)\approx1.091,\qquad
L_X^\mu(6)=0.750.
\]
Thus, both \(\mu_X(t)\) and \(L_X^\mu(t)\) are non-monotone.\\

\subsection{Monotonicity properties of \(FRAI\) and \(MRLAI\)}

Before comparing \(FRAI\) and \(MRLAI\), recall that the \(FRAI\)
function of \(X\) is denoted by \(L_X\) and is given by
\[
L_X(t)
=
\frac{t r_X(t)}{\displaystyle\int_0^t r_X(u)\,du}
=
\frac{tH_X'(t)}{H_X(t)},
\qquad 0<t<\tau_X.
\] A distribution is said to have increasing failure rate (\(IFR\)) if its failure rate is nondecreasing, and decreasing failure rate (\(DFR\)) if its failure rate is nonincreasing. Similarly, a distribution is said to have increasing ageing intensity (\(IAI\)) if its \(FRAI\) function is nondecreasing, and decreasing ageing intensity (\(DAI\)) if its \(FRAI\) function is nonincreasing.

Classical ageing theory relates failure-rate monotonicity to \(MRL\) monotonicity through the implications \(IFR\Rightarrow DMRL\) and \(DFR\Rightarrow IMRL\). It is therefore natural to ask whether analogous monotonicity implications hold between \(FRAI\) and \(MRLAI\). The following examples show that no such implication holds in general.

Before presenting the examples, we compare \(FRAI\) and \(MRLAI\) in a proportional hazards family. This comparison is natural because proportional changes in the hazard rate arise frequently in reliability and survival models. The following result establishes the invariance of the \(FRAI\) under a proportional-hazards transformation.
\begin{theorem}\label{thm:ph_frai_invariance}
Let \(\{X_\theta:\theta>0\}\) be a proportional hazards family with survival functions
\(\bar F_\theta(t)=\{\bar F_0(t)\}^{\theta}\), \(0\leq t<\tau_0\), and let
\(H_0(t)=-\log\bar F_0(t)\). If \(H_0\) is positive and differentiable on
\((0,\tau_0)\), then
\[
L_\theta(t)=L_0(t),\qquad 0<t<\tau_0.
\]
Hence, the \(FRAI\) function is invariant under the proportional hazards family.
\end{theorem}

\begin{proof}
For \(0<t<\tau_0\),
\[
H_\theta(t)
=
-\log\bar F_\theta(t)
=
-\log\{\bar F_0(t)\}^{\theta}
=
\theta H_0(t),
\]
and therefore
\[
H_\theta'(t)
=
\theta H_0'(t).
\]
Hence,
\[
L_\theta(t)
=
\frac{tH_\theta'(t)}{H_\theta(t)}
=
\frac{t\theta H_0'(t)}{\theta H_0(t)}
=
\frac{tH_0'(t)}{H_0(t)}
=
L_0(t),
\qquad 0<t<\tau_0.
\]
\end{proof}
In Theorem~\ref{cor:mrlaif_proportionality}, we have seen that the \(MRLAI\) function
is invariant within a proportional \(MRL\) family. However, this
invariance does not generally hold within a proportional hazards
family, as the following counterexample shows.

\begin{example}\label{cex:ph_mrlai_noninvariance}
Consider the Weibull family
\[
\bar F_\theta(t)=e^{-\theta t^\alpha},\qquad
t\geq0,\quad \alpha>1,\ \theta>0.
\]
Since \(H_\theta(t)=\theta t^\alpha\), we have
\[
L_\theta(t)=\frac{tH_\theta'(t)}{H_\theta(t)}=\alpha,
\qquad t>0,
\]
so the \(FRAI\) function is independent of \(\theta\). The corresponding
\(MRL\) function is
\[
\mu_\theta(t)
=e^{\theta t^\alpha}\int_t^\infty e^{-\theta u^\alpha}\,du.
\]
As \(t\downarrow0\) and \(t\to\infty\), respectively,
\[
\mu_\theta(t)\longrightarrow
\theta^{-1/\alpha}\Gamma\left(1+\frac{1}{\alpha}\right),
\qquad
\mu_\theta(t)\sim
\frac{1}{\theta\alpha t^{\alpha-1}},
\]
where the second relation follows from L'Hôpital's rule. Hence, for
\(\theta_1\neq\theta_2\),
\[
\lim_{t\downarrow0}
\frac{\mu_{\theta_1}(t)}{\mu_{\theta_2}(t)}
=
\left(\frac{\theta_2}{\theta_1}\right)^{1/\alpha},
\qquad
\lim_{t\to\infty}
\frac{\mu_{\theta_1}(t)}{\mu_{\theta_2}(t)}
=
\frac{\theta_2}{\theta_1}.
\]
Since \(\alpha>1\) and \(\theta_1\neq\theta_2\), these limits are
unequal, and hence \(\mu_{\theta_1}\) and \(\mu_{\theta_2}\) are not
proportional. Therefore, by
Theorem~\ref{cor:mrlaif_proportionality},
\(L_{\theta_1}^{\mu}\not\equiv L_{\theta_2}^{\mu}\). Thus, equality
of the \(FRAI\) functions does not imply equality of the corresponding
\(MRLAI\) functions.
\end{example}

\hspace*{.2in} An increasing \(MRLAI\) distribution may have a non-monotone \(FRAI\) function as the following example shows.
\begin{eg}\label{ex:imrlai-nonmonotone-frai}
 Consider a random variable having the $MRL$ function as
\[
\mu_X(t)=
\begin{cases}
1+\dfrac{t}{2}, & 0\leq t\leq 1,\\[5pt]
t+\dfrac{1}{2}, & t>1.
\end{cases}
\]
Then
\[
L_X^\mu(t)=
\begin{cases}
\dfrac{4+2t}{4+t}, & 0<t\leq 1,\\[7pt]
\dfrac{4t^2+2t}{2t^2+2t+1}, & t>1,
\end{cases}
\]
which is increasing in \(t>0\). Also,
\[
r_X(t)=
\begin{cases}
\dfrac{3}{t+2}, & 0<t<1,\\[7pt]
\dfrac{2}{t+\frac12}, & t>1,
\end{cases}
\]
and hence
\[
L_X(t)=
\begin{cases}
\dfrac{t}{(t+2)\log\left(\dfrac{t+2}{2}\right)}, & 0<t<1,\\[12pt]
\dfrac{\dfrac{4t}{2t+1}}
{3\log\left(\dfrac32\right)+2\log\left(\dfrac{2t+1}{3}\right)}, & t>1.
\end{cases}
\]
Since
\[
L_X(0.99)\approx0.8234,\qquad
L_X(1.01)\approx1.0879,\qquad
L_X(2)\approx0.7149,
\]
the \(FRAI\) function is non-monotone.\hfill$\Box$
\end{eg}
A  decreasing $MRLAI$ distribution may  lead to a non-monotonic $FRAI$ function as can be seen in the following example.
\begin{eg}\label{ex:dmrlai-nonmonotone-frai}
Consider a random variable $X$ having $\mu_{X}(t)=(1+t)^{-1},\;t \geq  0$ so that $$L_X^{\mu}(t) =\frac{t}{(1+t)\ln(1+t)},$$ 
 which is decreasing in $t > 0$. However, for this distribution, we have   \(r_{X}(t)=t(t+2)/(t+1)\) so that 
\begin{equation}\label{eqn:mrlai2}
    L_{X}(t) =\frac{t^{2} \left(t + 2\right)}{\left(t + 1\right) \left(\frac{t^2}{2} + t - \ln(t + 1)\right)},
\end{equation} which is non-monotone as Figure \ref{mrlai2} shows.\hfill$\Box$ 
\begin{figure}
	\centering
	\includegraphics[width=0.60\textwidth]{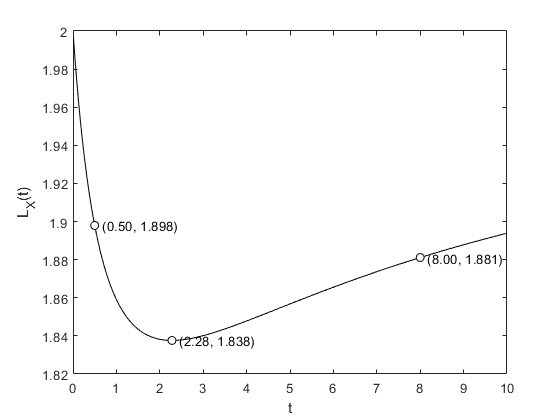}
	\caption{Non-monotone behaviour of the ageing intensity function \(L_X(t)\) given in \eqref{eqn:mrlai2}.}
	\label{mrlai2}
\end{figure}    
\end{eg} 
\hspace*{.2in} A distribution may simultaneously possess the increasing \(FRAI\) and increasing \(MRLAI\) properties as the following example shows.
\begin{eg}\label{ex:iai-imrlai}
 Let \(X\) have survival function
\[
\bar F_X(t)=\exp\{-t^{1/4}(1+t)^{1/2}\},\qquad t\geq0.
\]
Then the cumulative hazard is \(H_X(t)=-\log \bar F_X(t)=t^{1/4}(1+t)^{1/2}\), and hence
\[
L_X(t)=\frac{tH_X'(t)}{H_X(t)}
=\frac14+\frac{t}{2(1+t)},\qquad t>0,
\]
which is increasing in \(t\). Hence, \(X\) belongs to the \(IAI\) class. Further,
\[
\mu_X(t)=e^{t^{1/4}(1+t)^{1/2}}
\int_t^\infty \exp\{-u^{1/4}(1+u)^{1/2}\}\,du,
\qquad
L_X^\mu(t)=\frac{t\mu_X(t)}{\int_0^t\mu_X(u)\,du}.
\]
The graph of \(L_X^\mu(x/(1-x))\), \(0<x<1\), displayed in Figure~\ref{fig:F_6}, shows that it is increasing. Therefore, \(X\) is also \(IMRLAI\).\hfill$\Box$
\begin{figure}[ht]
    \centering
    \includegraphics[width=0.80 \textwidth]{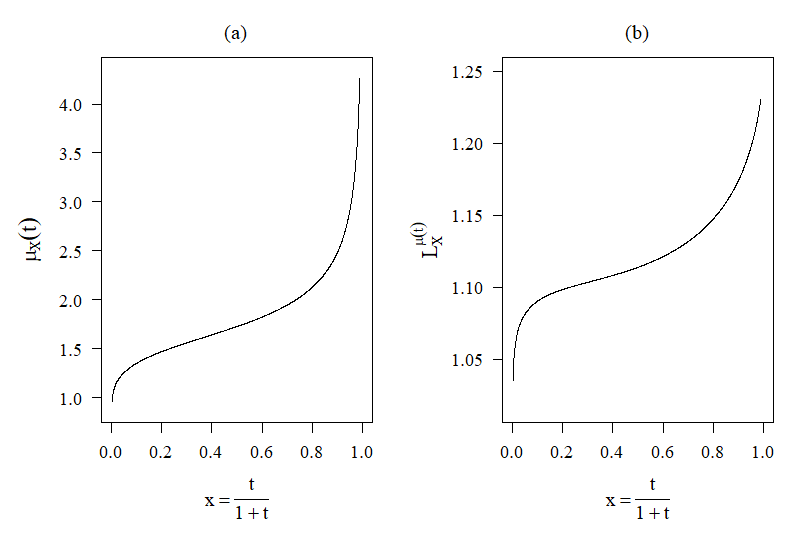}
    \caption{Plots of (a) the MRL function \(\mu_X(t)\) and (b) the \(MRLAI\) function \(L_X^\mu(t)\), both against \(x=t/(1+t)\), as discussed in Example~\ref{ex:iai-imrlai}.}
    \label{fig:F_6}
\end{figure}
\end{eg}
\hspace*{.2in} To demonstrate that the $FRAI$ function can increase while the $MRLAI$ function decreases, consider the following example.
\begin{eg}\label{ex:iai-dmrlai-uniform}
Consider a random variable \(X\sim U(0,1)\). Then the hazard rate is \(r_X(t)=(1-t)^{-1}\), \(0<t<1\), and the \(FRAI\) function is given by
\[
L_X(t)=\frac{t}{(1-t)(-\log(1-t))}, \qquad 0<t<1,
\]
which is increasing in \(t\). Since the MRL function of \(X\) is \(\mu_X(t)=(1-t)/2\), \(0<t<1\), the corresponding \(MRLAI\) function is
\[
L_X^{\mu}(t)=\frac{2(1-t)}{2-t}, \qquad 0<t<1,
\]
which is decreasing in \(t\). \hfill\(\Box\)
\end{eg}
That a \(DAI\) distribution may also exhibit decreasing \(MRLAI\) behaviour is discussed in the following example.
\begin{eg}\label{ex:dai-dmrlai}
 Let \(X\) have survival function
\[
\bar F_X(t)=
\begin{cases}
\exp\left(-\dfrac{2t-t^2}{4}\right), & 0\le t<1,\\[5pt]
0, & t\ge 1.
\end{cases}
\]
On the continuous part \(0<t<1\), we get
\[
L_X(t)=\frac{2(1-t)}{2-t},\qquad 0<t<1,
\]
which is decreasing in \(t\) on \((0,1)\). Moreover,
\[
\mu_X(t)
=
\frac{\displaystyle \int_t^1 \bar F_X(u)\,du}{\bar F_X(t)}
=
\frac{\sqrt{2}\displaystyle\int_{(t-1)/\sqrt{2}}^0 \exp(x^2/2)\,dx}
{\exp((t-1)^2/4)},\qquad 0\le t<1.
\]
Hence
\[
L_X^\mu(t)
=
\frac{t\mu_X(t)}
{\displaystyle\int_0^t\mu_X(u)\,du},
\qquad 0<t<1.
\]
A direct differentiation shows that \(L_X^\mu(t)\) is decreasing on \((0,1)\), as illustrated in Figure~\ref{fig:F_7}.\hfill$\Box$

\begin{figure}
    \centering
    \includegraphics[width=0.6\textwidth]{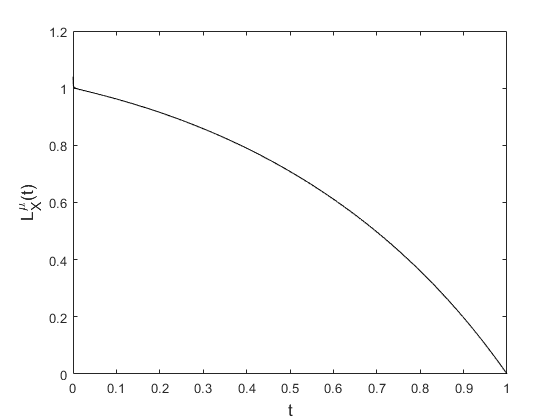}
    \caption{The \(MRLAI\) function for the distribution 
   discussed in Example \ref{ex:dai-dmrlai}.}
    \label{fig:F_7}
\end{figure}
\end{eg}
Below we see that a \(DAI\) distribution may be \(IMRLAI\).
\begin{eg}\label{ex:dai-imrlai}
Let \(X\) have hazard rate \(r_X(t)=2/(1+t)\), \(t\geq0\). Then \(H_X(t)=2\log(1+t)\), and hence
\[
L_X(t)=\frac{t}{(1+t)\log(1+t)}, \qquad t>0,
\]
which is decreasing in \(t>0\). Also, \(\bar F_X(t)=(1+t)^{-2}\), so that \(\mu_X(t)=1+t\), \(t\geq0\). Therefore, the corresponding \(MRLAI\) function is
\[
L_X^\mu(t)=\frac{2(t+1)}{t+2}, \qquad t>0,
\]
which is increasing in \(t\). Thus, \(X\) is \(DAI\) and \(IMRLAI\). \hfill\(\Box\)
\end{eg}
\hspace*{.2in} The following examples show that \(FRAI\) monotonicity need not carry over to \(MRLAI\).
\begin{eg}\label{ex:iai-nonmonotone-mrlai}
 To see that a distribution may be $IAI$ while having a non-monotone $MRLAI$,
consider a random variable $X$ with survival function given by
\[
\bar{F}_X(t)
= \frac{(t + 3\sqrt{t} + 2)\, e^{-t - 2\sqrt{t}}}{2},
\qquad t \ge 0.
\]
This gives
\(
\mu_X(t) = (\sqrt{t}+1)/(\sqrt{t}+2)\), 
so that the hazard rate becomes
\[
r_X(t)
= \frac{1+\mu'_X(t)}{\mu_X(t)}
= \frac{\sqrt{t}+2}{\sqrt{t}+1}
  \left(1+\frac{1}{2\sqrt{t}(\sqrt{t}+2)^2}\right),
\qquad t>0.
\]
This gives the $FRAI$ function as
\begin{eqnarray}
    L_X(t)&=& \frac{t\,r_X(t)}
       {\displaystyle \int_0^t r_X(u)\,du}\nonumber\\
&=& \frac{t\,\frac{\sqrt{t}+2}{\sqrt{t}+1}
       \left(1+\frac{1}{2\sqrt{t}(\sqrt{t}+2)^2}\right)}
       {\displaystyle 
        \int_0^t 
        \frac{\sqrt{u}+2}{\sqrt{u}+1}
        \left(1+\frac{1}{2\sqrt{u}(\sqrt{u}+2)^2}\right)\,du }\nonumber\\
&=&\frac{x^2\frac{x+2}{x+1}\left(1+\frac{1}{2x(x+2)^2}\right)}
{\displaystyle\int_0^x2s\,\frac{s+2}{s+1}\left(1+\frac{1}{2s(s+2)^2}\right) ds}.\label{eq:frai-transformed}
\end{eqnarray}
The last line follows by writing $t=x^2$ with 
\(x>0\).
Next, writing \(x =y/(1-y)\), we get the numerator of \eqref{eq:frai-transformed} as 

\[
x^2 \frac{x+2}{x+1} \Biggl( 1 + \frac{1}{2x(x+2)^2} \Biggr)
=
\frac{y^2}{(1-y)^2} (2-y) \Biggl( 1 + \frac{(1-y)^3}{2y(2-y)^2} \Biggr),
\]


while the denominator becomes
\[
\int_0^y \frac{2z(2-z)}{(1-z)^3} \left( 1 + \frac{(1-z)^3}{2z(2-z)^2} \right) dz.
\]

Consequently, for \(0<y<1\), we have
\[
L_X(t)
=
\frac{
\frac{y^2}{(1-y)^2} (2-y) \Bigl( 1 + \frac{(1-y)^3}{2y(2-y)^2} \Bigr)
}
{
\displaystyle
\int_0^y \frac{2z(2-z)}{(1-z)^3} \left( 1 + \frac{(1-z)^3}{2z(2-z)^2} \right) dz
}.
\]
Hence the $MRLAI$ function becomes
\[
L_X^{\mu}(t)
= \frac{t\,\mu_X(t)}{\displaystyle\int_0^t \mu_X(u)\,du}
= \frac{t(\sqrt{t}+1)}
       {(\sqrt{t}+2)\,
        \bigl(t - 2\sqrt{t}
        + 4\ln(\sqrt{t}+2)
        - 4\ln 2\bigr)},
\qquad t>0.
\]
\begin{figure}[ht]
\centering
\includegraphics[width=0.70\textwidth]{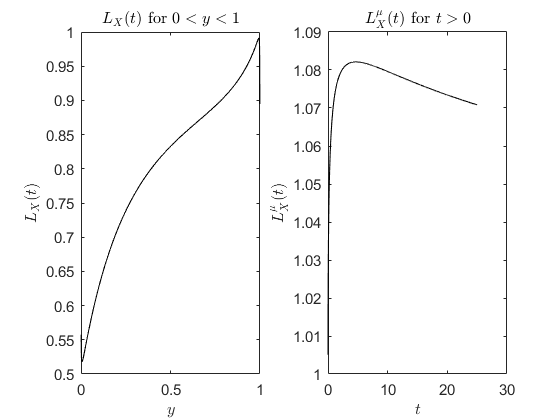}
\caption{The \(FRAI\) function is increasing, whereas the corresponding \(MRLAI\) function is non-monotone.}
\label{fig:AI_plot}
\end{figure}
It can be seen from Figure~\ref{fig:AI_plot} that $FRAI$ is increasing, whereas $MRLAI$ is non-monotone.   \hfill$\Box$
\end{eg}
\hspace*{.2in} A $DAI$ distribution can have a non-monotone $MRLAI$ function as the following example shows.
\begin{eg}\label{ex:dai-nonmonotone-mrlai}
 Consider a random variable having an Erlang distribution with density function given by $f_{X}(t) = 4te^{-2t}, \; t\geq 0$. Then we have, for $t>0$, 
    $$ L_{X}(t)  = \frac{tr_{X}(t)}{\int_{0}^{t} \,r_{X}(u)du} = \frac{4t^{2}}{\left(2t + 1\right) \left(2t - \ln\left(2t + 1\right)\right)},$$ 
which is decreasing in \(t\) on \((0,\infty)\). 
Further, the $MRL$ function is given by
\(
\mu_X(t) 
= (t + 1)/(2t + 1),\) 
so that 
\begin{equation*}
    L_X^{\mu}(t) = \frac{4t(t+1)}{(2t+1)(\ln(2t+1)+2t)},
\end{equation*}
which is non-monotone as Figure \ref{mrlai3} shows.
\begin{figure}
	\centering
	\includegraphics[height=5cm,keepaspectratio]{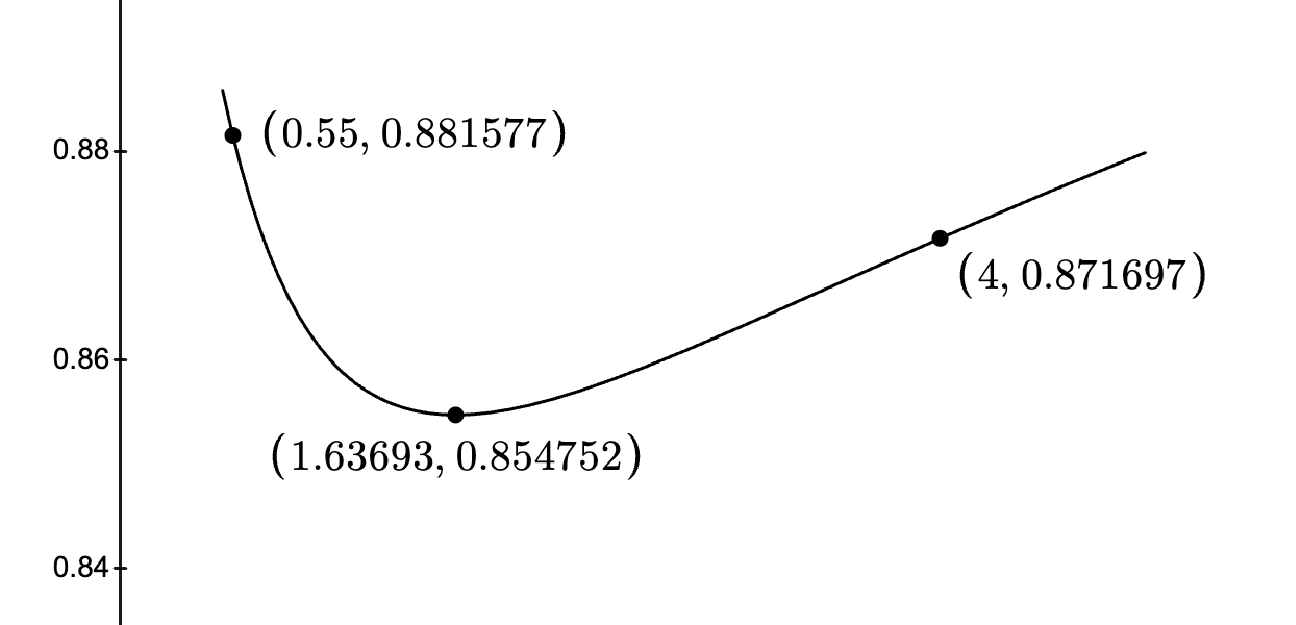}
	\caption{The \(MRLAI\) function in Example~\ref{ex:dai-nonmonotone-mrlai} is non-monotone.}
	\label{mrlai3}
\end{figure}   
\end{eg}
A distribution may have non-monotone \(FRAI\) and \(MRLAI\) both, as can be seen below.
\begin{eg}\label{ex:frai-mrlai-both-nonmonotone}
 Consider a lifetime random variable \(X\) with \(MRL\) function
\[
\mu_X(t)=
\begin{cases}
1+t, & 0\leq t\leq 1,\\[3pt]
(5-t)/2, & 1<t\leq 2,\\[3pt]
3/2, & t>2.
\end{cases}
\]
The corresponding \(MRLAI\) function is
\[
L_X^\mu(t)=
\begin{cases}
\dfrac{2(t+1)}{t+2}, & 0<t\leq 1,\\[8pt]
\dfrac{2t(5-t)}{10t-t^2-3}, & 1<t\leq 2,\\[8pt]
\dfrac{6t}{6t+1}, & t>2.
\end{cases}
\]
Since \(L_X^\mu(0.5)=1.2\), \(L_X^\mu(1)=1.3333\), \(L_X^\mu(2)=0.9231\), and \(L_X^\mu(3)=0.9474\), the \(MRLAI\) function is non-monotone. Further, using \(r_X(t)=\{1+\mu_X'(t)\}/\mu_X(t)\), the hazard rate is obtained as
\[
r_X(t)=
\begin{cases}
\dfrac{2}{1+t}, & 0<t<1,\\[8pt]
\dfrac{1}{5-t}, & 1<t<2,\\[8pt]
\dfrac{2}{3}, & t>2.
\end{cases}
\]
Thus the cumulative hazard function is
\[
H_X(t)=
\begin{cases}
2\ln(1+t), & 0<t\leq1,\\[5pt]
\ln\!\left(\dfrac{16}{5-t}\right), & 1<t\leq2,\\[10pt]
\ln\!\left(\dfrac{16}{3}\right)+\dfrac{2}{3}(t-2), & t>2,
\end{cases}
\]
and hence the \(FRAI\) function becomes
\[
L_X(t)=
\begin{cases}
\dfrac{t}{(1+t)\ln(1+t)}, & 0<t\leq1,\\[10pt]
\dfrac{t}{(5-t)\ln\!\left(\dfrac{16}{5-t}\right)}, & 1<t\leq2,\\[12pt]
\dfrac{2t}{3\ln\!\left(\dfrac{16}{3}\right)+2(t-2)}, & t>2.
\end{cases}
\]
Since \(L_X(0.5)\approx0.8221\), \(L_X(1.1)\approx0.1998\), and \(L_X(2.5)\approx0.8303\), the \(FRAI\) function is also non-monotone. Therefore, both \(L_X^\mu(t)\) and \(L_X(t)\) are non-monotone. \hfill\(\Box\)
\end{eg}
\subsection{Monotonicity Properties of $FRAI$ and $MRL$ functions}
Example~\ref{ex:iai-imrlai} shows that the \(FRAI\) and \(MRL\) functions may be increasing in \(t\) simultaneously. However, an increasing \(FRAI\) function may also be associated with a decreasing \(MRL\) function. To see this, consider the following example.
\begin{eg}
Suppose a lifetime random variable \(X\) has survival function \(\bar F_X(t)=\exp\{-t-t^2/2\}\), \(t\geq0\). Then \(r_X(t)=1+t\), and the corresponding \(FRAI\) function is \(L_X(t)=2(1+t)/(2+t)\), \(t>0\), which is strictly increasing; hence \(X\) is \(IAI\). Further,
\[
\mu_X(t)
=
\frac{\displaystyle\int_t^\infty\bar F_X(x)\,dx}{\bar F_X(t)}
=
\int_0^\infty e^{-u-tu-u^2/2}\,du
<
\int_0^\infty e^{-(1+t)u}\,du
=
\frac{1}{1+t}.
\]
Since \(\mu_X'(t)=r_X(t)\mu_X(t)-1=(1+t)\mu_X(t)-1<0\), the \(MRL\) function is strictly decreasing. Hence, \(X\) is \(DMRL\).\hfill\(\Box\)
\end{eg}
\hspace*{.2in} Further, an $IAI$ distribution may have a non-monotone $MRL$ function, as can be seen below.
\begin{eg}\label{A}
    Consider a lifetime random variable $X$ with survival function given by 
\[
\bar{F}_X(t) = \exp\!\left(- t^{1/2} e^{t/10}\right),\; t \ge 0.
\]
The corresponding hazard rate 
becomes
\[
r_X(t)= e^{0.1t}\!\left(\frac{1}{2}t^{-1/2} + 0.1 t^{1/2}\right),\;t>0,
\]
giving the corresponding $FRAI$ function as
$L_X(t)
= 0.5 + t/10$, which 
is strictly increasing in $t$. So, the distribution belongs to the $IAI$ class. The $MRL$ function is obtained as
\[
\mu_X(t)
= \frac{\displaystyle \int_t^\infty 
      \exp\!\left(-u^{1/2} e^{0.1u}\right)\,du}
       {\exp\!\left(-t^{1/2} e^{0.1t}\right)},
\]
which is non-monotone as Figure \ref{mrl1} shows.\hfill$\Box$
\end{eg}
\begin{figure}
	\centering
	\includegraphics[height=7cm,keepaspectratio]{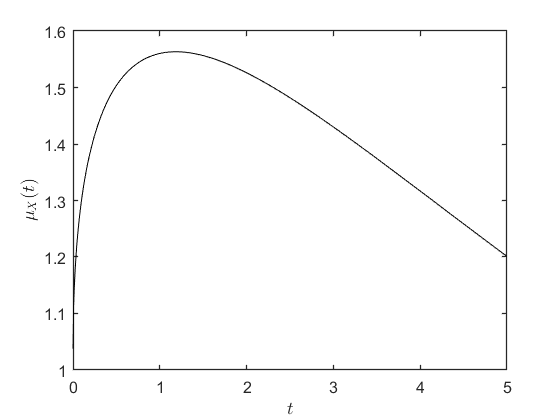}
     \caption{Non-monotone behaviour of the \(MRL\) function used in Example \ref{A}, for \(0\le t\le 5\).}
     \label{mrl1}
\end{figure}

To see that a distribution can be both $DAI$ and $DMRL$, we consider the following example. 
\begin{eg}\label{B}
Let a lifetime random variable $X$ have survival function
\[
\bar F_X(t)=
\begin{cases}
\exp\left(-\dfrac{2t-t^2}{4}\right), & 0\le t<1,\\[6pt]
0, & t\ge 1.
\end{cases}
\]
The corresponding hazard rate is obtained as 
\[r_X(t) = \frac{1-t}{2}, \quad 0 \le t < 1.\]
The corresponding $AI$ function comes out as
\[L_X(t) = \frac{2(1-t)}{2-t}, \quad 0 < t < 1.\]
So, $X$ belongs to the $DAI$ class. The corresponding $MRL$ function is given by
\begin{eqnarray*}
\mu_X(t)
&=& \exp\Bigl(\frac{2t - t^2}{4}\Bigr) \int_t^1 \exp\Bigl(-\frac{2u - u^2}{4}\Bigr)\,du \\
&=& \exp\Bigl(-\frac{(t-1)^2}{4}\Bigr) \int_t^1 \exp\Bigl(\frac{(u-1)^2}{4}\Bigr)\,du \\
&=& \exp\Bigl(-\frac{(t-1)^2}{4}\Bigr) \int_0^{1-t} \exp\Bigl(\frac{x^2}{4}\Bigr)\,dx,
\qquad 0\leq t<1,
\end{eqnarray*}
which is strictly decreasing as Figure~\ref{fig:MRL_AI} shows.\hfill$\Box$
\begin{figure}
    \centering
    \includegraphics[width=0.6\textwidth]{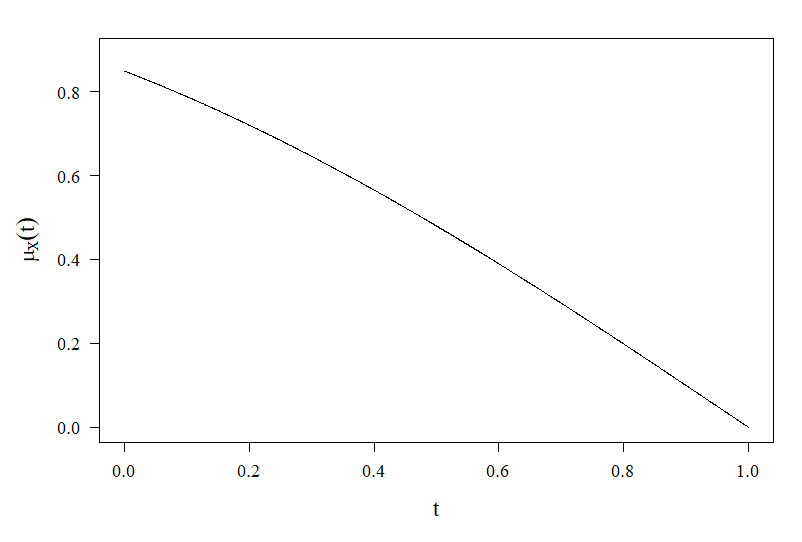}
    \caption{Plot of the \(MRL\) function considered in Example~\ref{B}.}
    \label{fig:MRL_AI}
\end{figure}  
\end{eg}
 The \(FRAI\) function may be decreasing when the corresponding \(MRL\) function is increasing, as we have seen in Example~\ref{ex:dai-imrlai}. To see that  a $DAI$ distribution may have a non-monotone $MRL$ function, consider the following example.
\begin{eg}\label{C} 
     Let a lifetime random variable $X$ have the survival function given by 
\[
\bar{F}_X(t) = \frac{1}{1 + t^{3/2}},\qquad t \ge 0.
\]
The corresponding hazard rate becomes 
\[
r_X(t)= \frac{\frac{3}{2}t^{1/2}}{1 + t^{3/2}},\qquad t>0,
\]
so that the $AI$ function comes out as
\[
L_X(t)= \frac{\frac{3}{2} t^{3/2}}{(1+t^{3/2})\ln(1+t^{3/2})},
\qquad t>0,
\]
which is strictly decreasing in $t$. So, the distribution belongs to the $DAI$ class. The $MRL$ function is obtained as
\begin{align*}
\mu_X(t)
  &= \bigl(1+t^{3/2}\bigr)\frac{2}{3}
     \int_{t^{3/2}}^{\infty}\frac{v^{-1/3}}{1+v}\,dv .
\end{align*}
\begin{figure}
    \centering
    \includegraphics[width=0.60 \textwidth]{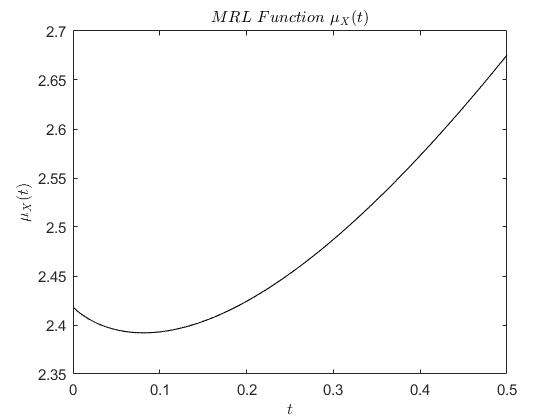}
    \caption{Plot of the $MRL$ function discussed in Example \ref{C}, 
    illustrating its non-monotone behavior.}
    \label{fig:MRL_DAI}
\end{figure}
From figure~\ref{fig:MRL_DAI}, we see that the above $\mu_X(t)$ first decreases and then increases, which is non-monotone, although $X$ is $DAI$. \hfill$\Box$
\end{eg}
\hspace*{.2in} In view of Example~\ref{ex:imrlai-nonmonotone-frai}, non-monotonicity of the \(FRAI\) function does not preclude an increasing \(MRL\) function. However, a non-monotone $FRAI$ may have decreasing $MRL$ also, as can be seen below.
\begin{eg}\label{D}
    Consider a random variable $X$ with survival function given by
\[
\bar{F}_X(t) = (1+t)^{-(1+t)}, \qquad t\geq 0,
\]
so that the corresponding hazard rate is
$r_X(t)
= 1+\ln(1+t), \; t>0.$
Thus, the $FRAI$ function becomes
\[
L_X(t)
= \frac{t\, r_X(t)}
        {\displaystyle\int_0^t r_X(u)\,du}
= \frac{t\,(1+\ln(1+t))}{(1+t)\ln(1+t)}, \; t>0,
\]
with the $MRL$ function as
\[
\mu_X(t)
= (1+t)^{1+t} \int_t^\infty (1+u)^{-(1+u)}\,du.
\]
To show that it is strictly decreasing, we use the monotone transformation
\(
t=x/(1-x),
\)
under which the \(MRL\) function becomes
\[
\mu_X\!\left(\frac{x}{1-x}\right)
= \left(1+\frac{x}{1-x}\right)^{1+\frac{x}{1-x}}
\int_{\frac{x}{1-x}}^{\infty} (1+u)^{-(1+u)}\,du,
\quad x\in(0,1).
\]
Figure~\ref{fig:AI_MRL} shows that the $FRAI$ function $L_X(t)$ is non-monotone, attaining a maximum near $t\approx 5$, 
while the $MRL$ function $\mu_X(t)$ is strictly decreasing.
\end{eg}
\begin{figure}
  \centering
   \includegraphics[width=0.70\textwidth]{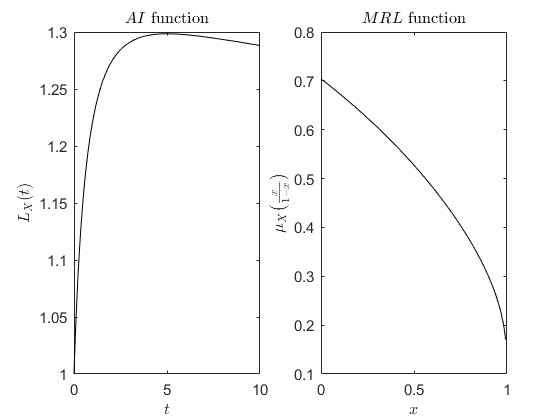}
  \caption{The \(FRAI\) function is non-monotone, whereas the corresponding \(MRL\) function is strictly decreasing, as discussed in Example \ref{D}.}
    \label{fig:AI_MRL}
\end{figure}  

A distribution may have non-monotone $FRAI$ and $MRL$ both, as can be seen below. 
 \begin{eg}\label{E}
Consider a lifetime random variable \(X\) with hazard rate
\[
r_X(t)=\frac{1}{(t+1)^2}+\frac{t}{10}+\frac{1}{10}, \qquad t\geq 0,
\]
having the \(FRAI\) function as
\[
L_X(t)
=
\frac{2\big((t+1)^3+10\big)}{(t+1)(t^2+3t+22)}, \qquad t>0,
\]
which is non-monotone, as Figure~\ref{fig:AI_MRL_E}(a) demonstrates. The corresponding \(MRL\) function is given by
\[
\mu_X(t)
=
\frac{\displaystyle \int_t^\infty
\exp\left[-\left(1-\frac{1}{u+1}+0.05u^2+0.1u\right)\right]\,du}
{\displaystyle
\exp\left[-\left(1-\frac{1}{t+1}+0.05t^2+0.1t\right)\right]},
\]
which is plotted in Figure~\ref{fig:AI_MRL_E}(b) to show that it is non-monotone. \hfill\(\Box\)

\begin{figure}[ht]
    \centering
    \includegraphics[width=0.85\textwidth]{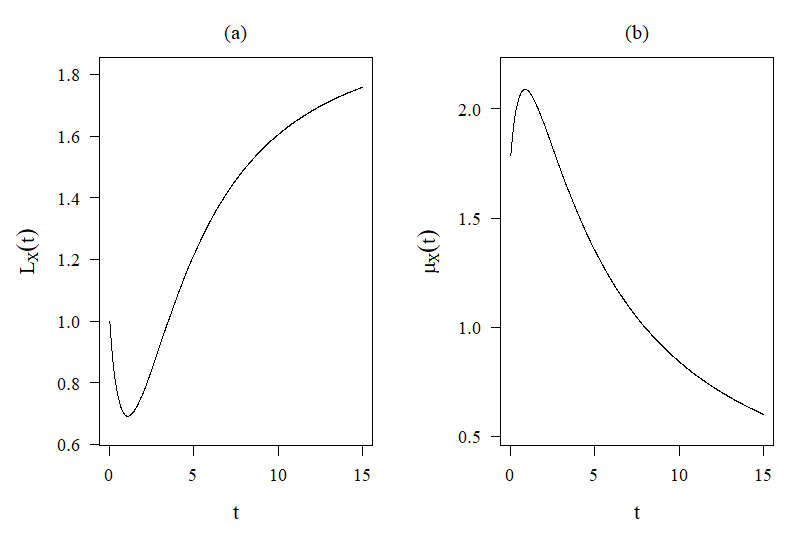}
    \caption{Plots of (a) the \(FRAI\) function \(L_X(t)\) and (b) the \(MRL\) function \(\mu_X(t)\), illustrating their non-monotone behaviour in Example~\ref{E}.}
    \label{fig:AI_MRL_E}
\end{figure}
\end{eg}
\section{Closure properties under different reliability operations}
  In this section, we study whether $DMRLAI$ and $IMRLAI$ classes are closed under different reliability operations, viz., mixture of distributions, convolution of distributions, and formation of coherent systems, in particular, series system, parallel system and $k$-out-of-$n$ system. By $k$-out-of-$n$ system, we mean a system which operates as long as at least $k$ out of the $n$ components forming the system work. Clearly, $1$-out-of-$n$ system is a parallel system, whereas an $n$-out-of-$n$ system is a series system.

\subsection{Closure properties of the $DMRLAI$ class} 
 Here, we establish the fact that the $DMRLAI$ class is not closed under any of the reliability operations discussed above. First, we demonstrate that the closure property under mixture does not apply to the $DMRLAI$ class. We must mention here that similar properties for $AI$ function based on hazard rate have been studied in Nanda et al.~\cite{nanda2007properties}.
 \begin{example}\label{closure-1}
    Let \(X_1\) and \(X_2\) be two random variables with respective $MRL$ functions
$\mu_{X_1}(t) =1/(1+t)$, and $\mu_{X_2}(t) =1/\{2(1+2t)\}$,
so that the corresponding $MRLAI$ functions become
\[
\begin{aligned}
L_{X_1}^\mu(t) 
&= \frac{t}{(1+t) \ln(1+t)}, \\[1mm]
L_{X_2}^\mu(t) 
&= \frac{2t}{(1+2t) \ln(1+2t)},
\end{aligned}
\]
which are decreasing in \(t > 0\). The survival functions of \(X_1\) and \(X_2\) are obtained as
\[
\begin{aligned}
\bar F_{X_1}(t) 
&= \frac{\mu_{X_1}(0)}{\mu_{X_1}(t)} \exp\Big(-\int_0^t \frac{du}{\mu_{X_1}(u)}\Big) \\
&= (1+t) \, e^{-(t + t^2/2)}, \\
\bar{F}_{X_2}(t) 
&= \frac{\mu_{X_2}(0)}{\mu_{X_2}(t)} \exp\Big(-\int_0^t \frac{du}{\mu_{X_2}(u)}\Big) \\
&= (1+2t) \, e^{-(2t + 2t^2)}.
\end{aligned}
\]
Now, consider the mixture random variable \(X\) having survival function
\[
\bar{F}_X(t) = 0.2\, \bar F_{X_1}(t) + 0.8\, \bar{F}_{X_2}(t) 
= 0.2(1+t) \, e^{-(t + t^2/2)} + 0.8(1+2t) \, e^{-(2t + 2t^2)}.
\]
This has the $MRL$ function as
\[
\begin{aligned}
\mu_X(t) 
&= \frac{e^{t+3t^2/2} + 2}{(t+1)\, e^{t+3t^2/2} + 8t + 4}.
\end{aligned}
\]
The corresponding $MRLAI$ function comes out as
\[
\begin{aligned}
L_X^\mu(t) 
&= \frac{t \left(e^{t+3t^2/2} + 2\right)}
{\left((t+1)e^{t+3t^2/2} + 8t + 4\right)
\displaystyle\int_0^t
\frac{e^{u+3u^2/2} + 2}
{(u+1)e^{u+3u^2/2} + 8u + 4}\,du},
\end{aligned}
\]
which is non-monotonic as Figure~\ref{fig:F_25} shows. Hence, the class of $DMRLAI$ distributions is not closed under mixtures.\hfill$\Box$       
\begin{figure}
    \centering
    \includegraphics[width=0.7\textwidth]{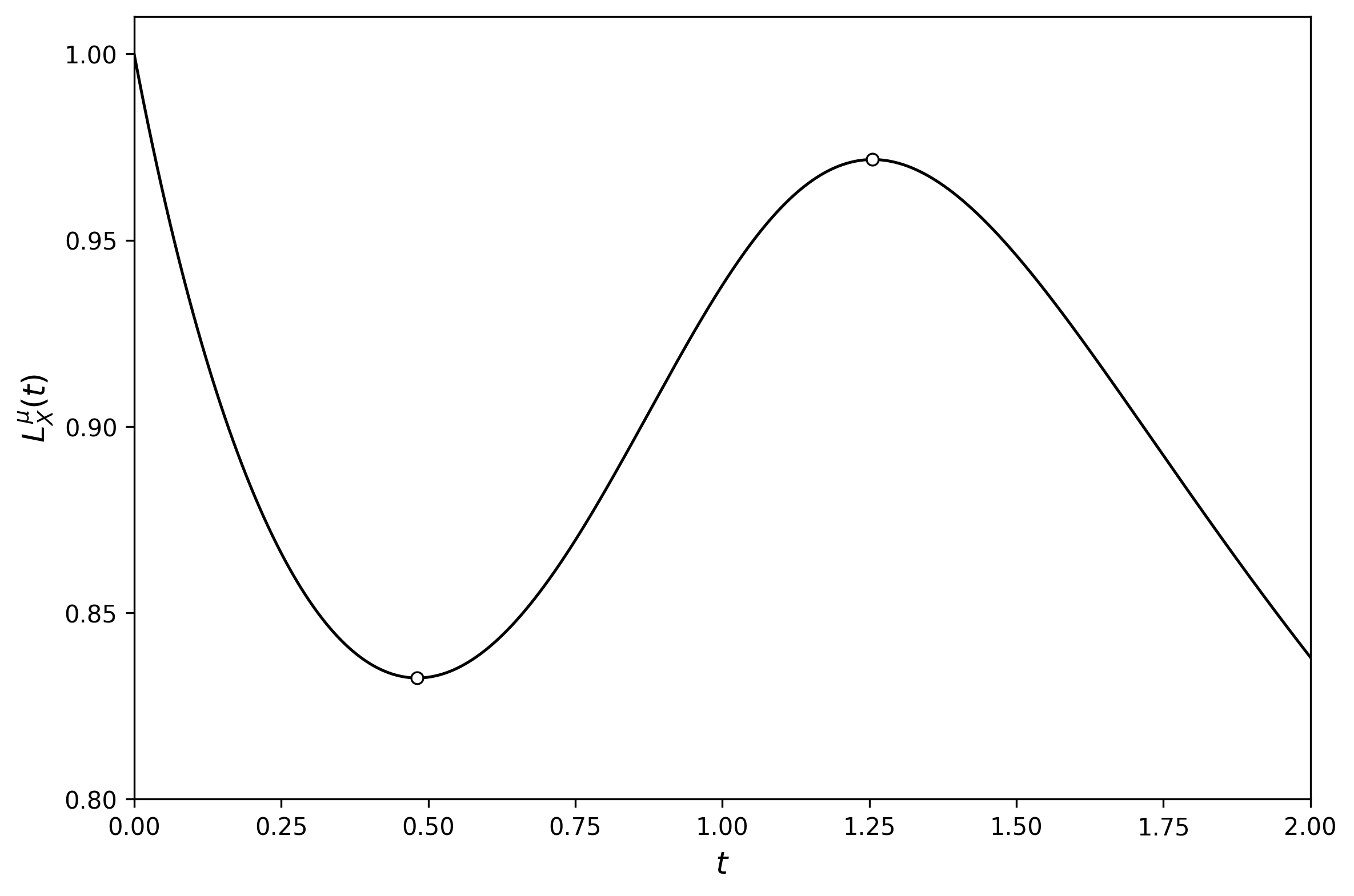}
    \caption{MRLAI function of the mixture random variable $X$ (discussed in Counterexample \ref{closure-1}), exhibiting non-monotonic behaviour.}
    \label{fig:F_25}
\end{figure}
\end{example}

 Order statistics find practical applications in analysing the sequence of failures in components or systems. There is a one-to-one correspondence between the order statistics and the $k$-out-of-$n$ systems. Actually, the $k$-th order statistic represents the life of an $(n-k+1)$-out-of-$n$ system. Below, we present a counterexample to demonstrate that the $DMRLAI$ class is not closed under the formation of a parallel system.
\begin{example}\label{ex:dmrlai-not-closed-parallel}
Let \(X_i,\ i=1,2,\) be two i.i.d.\ non-negative random variables with \(MRL\) function \(\mu_{X_i}(t)=(1+t)^{-1/10}\), \(t\geq0\). The corresponding survival function is
\[
\bar F_{X_i}(t)
=
(1+t)^{1/10}
\exp\left\{-\frac{10}{11}\left((1+t)^{11/10}-1\right)\right\},
\qquad t\geq0.
\]
Hence
\[
L_{X_i}^{\mu}(t)
=
\frac{9t}{10\left\{(1+t)-(1+t)^{1/10}\right\}},
\qquad t>0.
\]
Moreover,
\[
\frac{d}{dt}L_{X_i}^{\mu}(t)
=
\frac{9\left\{1-\frac{9}{10}(1+t)^{1/10}
-\frac{1}{10}(1+t)^{-9/10}\right\}}
{10\left\{(1+t)-(1+t)^{1/10}\right\}^{2}}<0,\qquad t>0,
\]
since \(\frac{9}{10}(1+t)^{1/10}+\frac{1}{10}(1+t)^{-9/10}>1\), for \(t>0\). Thus each \(X_i\) belongs to the \(DMRLAI\) class. Now, write \(X=\max\{X_1,X_2\}\). Since \(X_1\) and \(X_2\) are i.i.d., the survival function of \(X\) is
\[
\bar F_X(t)
=
2\bar F_{X_1}(t)-\{\bar F_{X_1}(t)\}^2,
\qquad t\geq0.
\]
Equivalently,
\[
\bar F_X(t)
=
2(1+t)^{1/10}e^{-\frac{10}{11}\{(1+t)^{11/10}-1\}}
-
(1+t)^{1/5}e^{-\frac{20}{11}\{(1+t)^{11/10}-1\}}.
\]
Therefore, the \(MRL\) function of \(X\) is
\[
\mu_X(t)
=
\frac{
2e^{-\frac{10}{11}\{(1+t)^{11/10}-1\}}
-
\frac{10}{11}e^{20/11}\left(\frac{20}{11}\right)^{-12/11}
\Gamma\!\left(\frac{12}{11},\frac{20}{11}(1+t)^{11/10}\right)
}
{
2(1+t)^{1/10}e^{-\frac{10}{11}\{(1+t)^{11/10}-1\}}
-
(1+t)^{1/5}e^{-\frac{20}{11}\{(1+t)^{11/10}-1\}}
},
\qquad t\geq0,
\]
where \(\Gamma(a,x)=\int_x^\infty u^{a-1}e^{-u}\,du\). Hence
\[
L_X^\mu(t)
=
\frac{t\mu_X(t)}
{\displaystyle\int_0^t\mu_X(u)\,du},
\qquad t>0.
\]
Numerical integration gives \(L_X^\mu(0.1)\approx0.9701\), \(L_X^\mu(2)\approx0.8528\), and \(L_X^\mu(20)\approx0.8981\), showing that \(L_X^\mu(t)\) is non-monotone. Hence, the $DMRLAI$ class is not closed under the formation of a parallel system showing that the said class is not closed under the formation of a $k$-out-of-$n$ system and coherent system of $i.i.d.$ components. This immediately tells that the $DMRLAI$ class is not closed under the formation of parallel system (and hence of $k$-out-of-$n$ systems and coherent system of independent but not necessarily identical) components.\hfill$\Box$
\end{example}

The following counterexample shows that the \(DMRLAI\) class is not closed under the formation of a series system.

\begin{example}\label{ex:dmrlai-not-closed-series}
Let \(X_1\) and \(X_2\) be i.i.d. lifetime random variables with common \(MRL\) function
\[
\mu_X(t)=
\begin{cases}
2, & 0\leq t\leq \dfrac52,\\[3pt]
2-\dfrac{3}{125}\left(t-\dfrac52\right), & \dfrac52<t\leq 5,\\[8pt]
\dfrac{97}{125}\left(\dfrac{15}{2}-t\right), & 5<t<\dfrac{15}{2}.
\end{cases}
\]
Put \(\tau_X=15/2\) and \(x=t/\tau_X\). Then the break-points \(t=5/2\) and \(t=5\) correspond to \(x=1/3\) and \(x=2/3\), respectively. The component \(MRLAI\) function is
\[
L_X^\mu(\tau_X x)=
\begin{cases}
1, & 0<x\leq \dfrac13,\\[6pt]
\dfrac{2x(103-9x)}{206x-9x^2-1}, & \dfrac13<x\leq \dfrac23,\\[10pt]
\dfrac{1746x(1-x)}{1746x-873x^2-379}, & \dfrac23<x<1.
\end{cases}
\]
Figure~\ref{fig:dmrlai-series-counterexample}(a) shows that \(L_X^\mu(\tau_X x)\) is nonincreasing on \(0<x<1\), and hence each component belongs to the \(DMRLAI\) class. Now, let \(X_{1:2}=\min\{X_1,X_2\}\) denote the lifetime of the two-component series system. The survival function of \(X_{1:2}\), on the transformed scale, is
\[
\bar F_{X_{1:2}}(\tau_X x)=
\begin{cases}
e^{-15x/2}, & 0\leq x\leq \dfrac13,\\[6pt]
e^{-5/2}\left(\dfrac{103-9x}{100}\right)^{244/3}, & \dfrac13<x\leq \dfrac23,\\[10pt]
e^{-5/2}\left(\dfrac{97}{100}\right)^{244/3}\{3(1-x)\}^{56/97}, & \dfrac23<x<1.
\end{cases}
\]
Writing \(a=97/100\) and \(z_x=(103-9x)/100\), the corresponding \(MRL\) function of \(X_{1:2}\) is
\[
\mu_{X_{1:2}}(\tau_X x)=
\begin{cases}
1-e^{-5/2+15x/2}
+\dfrac{15}{2}e^{-5/2+15x/2}
\left[
\dfrac{100}{741}\{1-a^{247/3}\}
+\dfrac{97}{459}a^{244/3}
\right], & 0<x\leq \dfrac13,\\[14pt]
\dfrac{15}{2}z_x^{-244/3}
\left[
\dfrac{100}{741}\{z_x^{247/3}-a^{247/3}\}
+\dfrac{97}{459}a^{244/3}
\right], & \dfrac13<x\leq \dfrac23,\\[14pt]
\dfrac{485}{102}(1-x), & \dfrac23<x<1.
\end{cases}
\]
Therefore,
\[
L_{X_{1:2}}^\mu(\tau_X x)
=
\frac{x\,\mu_{X_{1:2}}(\tau_X x)}
{\displaystyle\int_0^x \mu_{X_{1:2}}(\tau_X y)\,dy},
\qquad 0<x<1.
\]
Figure~\ref{fig:dmrlai-series-counterexample}(b) shows that \(L_{X_{1:2}}^\mu(\tau_X x)\) is non-monotone. Hence, the \(DMRLAI\) class is not closed under the formation of a series system. \hfill\(\Box\)
\end{example}

\begin{figure}[htbp]
\centering
\includegraphics[width=0.92\textwidth]{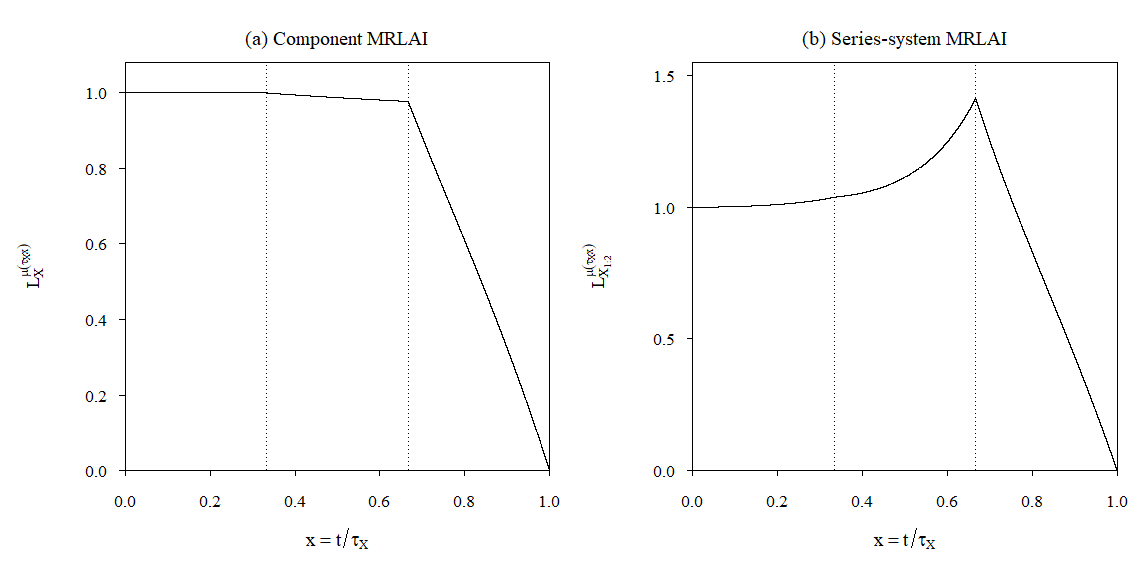}
\caption{Plots of the \(MRLAI\) functions against the transformed scale \(x=t/\tau_X\), where \(\tau_X=15/2\): (a) the component \(MRLAI\) function \(L_X^\mu(\tau_X x)\); (b) the series-system \(MRLAI\) function \(L_{X_{1:2}}^\mu(\tau_X x)\).}
\label{fig:dmrlai-series-counterexample}
\end{figure}

Although the \(DMRLAI\) class is not closed under the formation of series systems in general, the following result shows that preservation holds within the linear-\(MRL\) family.
\begin{theorem}\label{thm:series-linear-mrl}
Let \(n\geq1\), and let \(X_1,\ldots,X_n\) be i.i.d. absolutely continuous lifetime random variables over the support with common differentiable survival function \(\bar F_X\), and \(MRL\) function \(\mu_X(t)=a+bt\), \(0\leq t<\tau_X\), where \(a>0\) and \(b>-1\). If \(X_{1:n}=\min\{X_1,\ldots,X_n\}\), then \(L^\mu_{X_{1:n}}(t)=L^\mu_X(t)\), \(0<t<\tau_X\).
\end{theorem}

\begin{proof}
Let \(M_X(t)=\int_t^\infty\bar F_X(u)\,du\), \(0\leq t<\tau_X\). Then, \(\mathbb{E}(X)=M_X(0)=a<\infty\). If \(\tau_{X}<\infty\), we have \(\bar F_X(u)=0\) for \(u>\tau_X\), and
\[
\begin{aligned}
0\leq\mu_X(t)
&=
\frac{\displaystyle\int_t^{\tau_X}\bar F_X(u)\,du}
{\bar F_X(t)}
\leq
\frac{(\tau_X-t)\bar F_X(t)}
{\bar F_X(t)}
=
\tau_X-t.
\end{aligned}
\]
It follows that \(\mu_X(t)\to0\) as \(t\uparrow\tau_X\). If \(b\geq0\), then \(\mu_X(t)=a+bt\geq a>0\), which is a contradiction. Thus, \(\tau_X\) cannot be finite when \(b\geq0\).

Now suppose that \(-1<b<0\). Since \(\mu_X(t)=a+bt>0\) for every \(0\leq t<\tau_X\), we have \(\tau_X\leq-a/b(<\infty)\) and \(0=\lim_{t\uparrow\tau_X}\mu_X(t)=a+b\tau_X,\) so that \(\tau_X=-a/b\). In either case, set \(D_n=n(1+b)-b=1+(n-1)(1+b)>0\). It is known that \(\bar F_{X_{1:n}}(t)=\{\bar F_X(t)\}^{n}\), \(t\geq0\). Thus,
\[
\int_0^\infty\{\bar F_X(u)\}^{n}\,du
\leq
\int_0^\infty\bar F_X(u)\,du
=
a,
\]
so the mean and the \(MRL\) function of \(X_{1:n}\) are finite.

Differentiating \(M_X(t)=\mu_X(t)\bar F_X(t)\) gives \(-\bar F_X(t)=b\bar F_X(t)+\mu_X(t)\bar F_X'(t)\) for all \(t\in(0,\tau_X)\). Therefore, \(\mu_X(t)\bar F_X'(t)=-(1+b)\bar F_X(t)\) for all \(t\in(0,\tau_X)\). Define \(G_n(t)=\mu_X(t)\{\bar F_X(t)\}^{n}/D_n\), \(0\leq t<\tau_X\). Then we have
\[
\begin{aligned}
G_n'(t)
&=
\frac{
b\{\bar F_X(t)\}^{n}
+
n\mu_X(t)\{\bar F_X(t)\}^{n-1}\bar F_X'(t)}
{D_n}\\
&=
\frac{b-n(1+b)}{D_n}\{\bar F_X(t)\}^{n}
=
-\{\bar F_X(t)\}^{n}
\end{aligned}
\]
for all \(t\in(0,\tau_X)\).

If \(\tau_{X}<\infty\), we get \(\mu_X(t)\to0\) as \(t\uparrow\tau_X\) and \(G_n(t)\to0\) as \(t\uparrow\tau_{X}\). If \(\tau_{X}\) is infinite, we have \(M_X(t)=\int_t^\infty\bar F_X(u)\,du\longrightarrow0\) as \(t\to\infty\). Furthermore,
\[
0
\leq
\mu_X(t)\{\bar F_X(t)\}^{n}
=
M_X(t)\{\bar F_X(t)\}^{n-1}
\leq
M_X(t),
\]
and hence \(G_n(t)\to0\) as \(t\to\infty\). Again, for \(0\leq t<s<\tau_X\), we have \(\int_t^s\{\bar F_X(u)\}^{n}\,du=G_n(t)-G_n(s)\). Letting \(s\uparrow\tau_{X}\), we obtain
\[
\int_t^\infty\{\bar F_X(u)\}^{n}\,du
=
\frac{\mu_X(t)\{\bar F_X(t)\}^{n}}{D_n},
\qquad 0\leq t<\tau_X.
\]
Since \(\bar F_{X_{1:n}}(t)=\{\bar F_X(t)\}^{n}\), it follows that
\[
\begin{aligned}
\mu_{X_{1:n}}(t)
&=
\frac{\displaystyle\int_t^\infty
\bar F_{X_{1:n}}(u)\,du}
{\bar F_{X_{1:n}}(t)}
=
\frac{\mu_X(t)}{D_n},
\qquad 0\leq t<\tau_X.
\end{aligned}
\]
Thus, \(\mu_{X_{1:n}}\) is a positive constant multiple of \(\mu_X\). Therefore,
\[
\begin{aligned}
L^\mu_{X_{1:n}}(t)
&=
\frac{t\mu_{X_{1:n}}(t)}
{\displaystyle\int_0^t\mu_{X_{1:n}}(u)\,du}
=
\frac{t\mu_X(t)}
{\displaystyle\int_0^t\mu_X(u)\,du}
=
L^\mu_X(t),
\qquad 0<t<\tau_X.
\end{aligned}
\]
\end{proof}

Thus, within the linear-\(MRL\) family, the \(DMRLAI\) and \(IMRLAI\) properties are preserved under series systems of i.i.d.\ components.

Below we show that the closure property under convolution does not apply to the $DMRLAI$ class.
\begin{example}\label{imrlai-convolution}
Let $X_{1}$ and $X_{2}$ be $i.i.d.$ standard exponential random variables, which may be considered as $DMRLAI$. Then the $MRLAI$ function is 
given by $L_{X_{1}}^{\mu}(t)=L_{X_{2}}^{\mu}(t)=1$, for all $t>0$. Now, take $X_{c}=X_{1}+X_{2}$. The survival function of $X_{c}$ is given by
$\bar{F}_{X_c}(t)=(t+1)e^{-t}$, whereas the $MRL$ function comes out as
$\mu_{X_c}(t)=(t+2)/(t+1)$.
Then the $MRLAI$ function of $X_{c}$ becomes
\[
L_{X_{c}}^{\mu}(t)
=\frac{t(t+2)}{(t+1)\left(\ln(t+1)+t\right)}.
\]
which is non-monotone as
$L_{X_{c}}^{\mu}(0.20)=0.9590531$,
$L_{X_{c}}^{\mu}(3.0)=0.8549358$, and
$L_{X_{c}}^{\mu}(10.0)=0.8799147$.
Thus, we see that the $DMRLAI$ class is not closed under convolution of $i.i.d.$ components. This immediately tells that the $DMRLAI$ class is not closed under convolution of independent but bot necessarily identical components.\hfill$\Box$
\end{example}
 The above results are summarised in Table~\ref{tab:dmrlai_nonclosure}.
 \begin{table}[h]
\centering
\caption{Closure status of the \(DMRLAI\) class under selected reliability operations.}
\label{tab:dmrlai_nonclosure}
\begin{tabular}{|l|c|}
\hline
Operation & $DMRLAI$ class \\ \hline
Mixture & Not closed \\ \hline
Convolution & Not closed \\ \hline
Parallel system & Not closed \\ \hline
Series system & Not closed \\ \hline
\end{tabular}
\end{table}
\subsection{Closure properties of the \(IMRLAI\) class}
We next consider the \(IMRLAI\) class. The following counterexamples show that this class is not, in general, closed under the reliability operations considered above. We first show that it is not closed under mixtures; that is, a mixture of two \(IMRLAI\) distributions need not be \(IMRLAI\).
\begin{example}\label{closure-2}
Let $X_1$ and $X_2$ be two random variables with respective $MRL$ functions given by
$\mu_{X_1}(t) = 1 + 8t$ and $\mu_{X_2}(t) = 1 + 0.1t$, for $t \ge 0$.
The corresponding survival functions are obtained as
$\bar F_{X_1}(t) = (8t + 1)^{-9/8}$ and
$\bar{F}_{X_2}(t) = (0.1t + 1)^{-11}$, respectively, which yield the corresponding
$MRLAI$ functions as
$L_{X_1}^{\mu}(t) = (1 + 8t)/(1 + 4t)$ and
$L_{X_2}^{\mu}(t) = (1 + 0.1t)/(1 + 0.05t)$, which are both increasing.
Consider now a random variable $X$ with survival function given by
\[
\bar{F}_X(t)
= 0.2(8t + 1)^{-9/8} + 0.8(0.1t + 1)^{-11},
\]
so that the $MRL$ function comes out as
\begin{align*}
\mu_X(t) &= \frac{0.2 (8t + 1)^{-1/8} + 0.8 (0.1t + 1)^{-10}}
        {0.2 (8t + 1)^{-9/8} + 0.8 (0.1t + 1)^{-11}}.
\end{align*}
The corresponding $MRLAI$ function takes the form
\[
L_X^{\mu}(t)
=
\frac{
\displaystyle
\frac{0.2(8t + 1)^{-1/8} + 0.8(0.1t + 1)^{-10}}
     {0.2(8t + 1)^{-9/8} + 0.8(0.1t + 1)^{-11}}
}{
\displaystyle
\frac{1}{t}\int_0^t
\frac{0.2(8u + 1)^{-1/8} + 0.8(0.1u + 1)^{-10}}
     {0.2(8u + 1)^{-9/8} + 0.8(0.1u + 1)^{-11}}
\,du
}.
\]
Figure~\ref{fig:MRLAI_F13} shows that $L_X^{\mu}(t)$ is non-monotonic. Consequently, the $IMRLAI$ class is not closed under mixture. 
\begin{figure}
    \centering
    \includegraphics[width=0.7\textwidth]{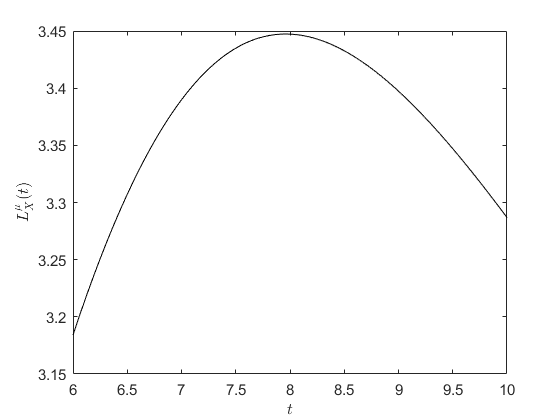}
    \caption{Non-monotone behaviour of the \(MRLAI\) function of the mixture random variable \(X\) (see Counterexample \ref{closure-2}).}
    \label{fig:MRLAI_F13}
\end{figure}
\hfill$\Box$
\end{example}

That the $IMRLAI$ class is not closed under convolution can be established by using Counterexample~ \ref{imrlai-convolution}, since exponential distribution may be considered as an $IMRLAI$ as well. \\
 \hspace*{.2in}Below, we present a counterexample to demonstrate that the $IMRLAI$ class is not closed under the formation of parallel systems.
\begin{example}\label{closure-3}
 Let $X_{1}$ and $X_{2}$ be two $i.i.d.$ random variables with survival function 
\(
\bar F_{X_1}(t) = (1+t)^{-2}\), for \( t \ge 0.
\)
The corresponding $MRL$ function is $\mu_{X_i}(t) = 1 + t,$ giving that $L_{X_i}^{\mu}(t) = (2(1+t))/(2+t),$
which is increasing in $t \ge 0$. Now, write $X = \max\{X_1, X_2\}.$
The survival function of $X$ is
\[
\begin{aligned}
\bar{F}_X(t) 
&= 1 - \left(1 - \frac{1}{(t+1)^2}\right)^2.
\end{aligned}
\]
This gives the $MRL$ function of $X$ as
\[
\mu_X(t) 
= \frac{(t+1)(6t^2 + 12t + 5)}{3(2t^2 + 4t + 1)}.
\]
Further, note that
\[
\begin{aligned}
\int_0^t \mu_X(u)\,du 
&= \frac{\ln(2t^2 + 4t + 1) + 3t(t+2)}{6}.
\end{aligned}
\]
so that
\[
L_X^\mu(t)
= \frac{2t (t+1) (6t^2 + 12t + 5)}{(2t^2 + 4t + 1)\left[\ln(2t^2 + 4t + 1) + 3t(t+2)\right]} , 
\]which is non-monotone, as is clear from Figure~\ref{fig:F_17}. Hence, the $IMRLAI$ class is not closed under the formation of a parallel system. This immediately tells us that the $IMRLAI$ class is not closed under the formation of a $k$-out-of-$n$ system and hence of coherent systems of $i.i.d.$ components. This immediately tells that the $IMRLAI$ class is not closed under the formation of parallel systems (and hence of $k$-out-of-$n$ systems and coherent systems) of independent but not necessarily identical components.
\begin{figure}
    \centering
    \includegraphics[width=0.70\textwidth]{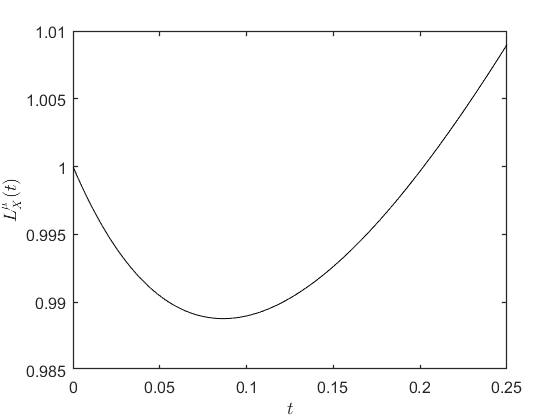}
    \caption{Non-monotone behaviour of the \(MRLAI\) function of the parallel-system lifetime, as discussed in Counterexample \ref{closure-3}.}
    \label{fig:F_17}
\end{figure}

 \end{example} 

 The following counterexample shows that the \(IMRLAI\) class is not closed under the formation of a series system.

\begin{example}\label{cex:imrlai-not-closed-series}
Let \(X_1\) and \(X_2\) be i.i.d. lifetime random variables with common \(MRL\) function
\[
\mu_X(t)=1+5\sqrt{t},\qquad t\geq0.
\]
Then
\[
L_X^\mu(t)
=
\frac{t(1+5\sqrt{t})}{\displaystyle\int_0^t(1+5\sqrt{u})\,du}
=
\frac{1+5\sqrt{t}}{1+\frac{10}{3}\sqrt{t}},
\qquad t>0,
\]
which is increasing in \(t\). Thus each component belongs to the \(IMRLAI\) class. The survival function corresponding to \(\mu_X(t)\) is
\[
\bar F_X(t)
=
(1+5\sqrt{t})^{-23/25}
\exp\left\{-\frac{2}{5}\sqrt{t}\right\},
\qquad t\geq0.
\]
Now let \(X_{1:2}=\min\{X_1,X_2\}\) denote the lifetime of the two-component series system. Then
\[
\bar F_{X_{1:2}}(t)
=
\{\bar F_X(t)\}^2
=
(1+5\sqrt{t})^{-46/25}
\exp\left\{-\frac{4}{5}\sqrt{t}\right\},
\qquad t\geq0.
\]
Hence
\[
\mu_{X_{1:2}}(t)
=
\frac{\displaystyle\int_t^\infty
(1+5\sqrt{u})^{-46/25}
\exp\left\{-\frac{4}{5}\sqrt{u}\right\}\,du}
{(1+5\sqrt{t})^{-46/25}
\exp\left\{-\frac{4}{5}\sqrt{t}\right\}},
\]
and
\[
L_{X_{1:2}}^\mu(t)
=
\frac{t\mu_{X_{1:2}}(t)}
{\displaystyle\int_0^t \mu_{X_{1:2}}(u)\,du},
\qquad t>0.
\]
Numerical integration gives
\[
L_{X_{1:2}}^\mu(1)\approx1.5577,\qquad
L_{X_{1:2}}^\mu(4)\approx1.5778,\qquad
L_{X_{1:2}}^\mu(25)\approx1.5649.
\]
Thus \(L_{X_{1:2}}^\mu(t)\) is not nondecreasing. Hence, the \(IMRLAI\) class is not closed under the formation of a series system. \hfill\(\Box\)
\end{example}
 The above results are summarised in Table~\ref{tab:imrlai_nonclosure}.

\medskip

 \begin{table}[h]
\centering
\caption{Closure status of the \(IMRLAI\) class under selected reliability operations.}
\label{tab:imrlai_nonclosure}
\begin{tabular}{|l|l|}
\hline
Operation & $IMRLAI$ class \\ \hline
Mixture & Not closed \\ \hline
Convolution & Not closed \\ \hline
Parallel system & Not closed \\ \hline
Series system & Not closed \\ \hline
\end{tabular}
\end{table}

\section{Some properties of MRLAI order}
As we have discussed earlier, $MRLAI$ function is an important indicator of ageing. This immediately tells us that we should be able to compare two life distributions in terms of their $MRLAI$ functions. To describe this, we define a probabilistic order based on the $MRLAI$ function $L_X^{\mu}(t)$  as follows.

For two lifetime random variables \(X\) and \(Y\), let
\(\tau_{XY}=\min\{\tau_X,\tau_Y\}\).

\begin{definition}
A lifetime random variable \(X\) is said to be smaller than \(Y\) in the \(MRLAI\) order, denoted by \(X\leq_{MRLAI}Y\), if \(L_X^\mu(t)\leq L_Y^\mu(t)\) for all \(0<t<\tau_{XY}\). \hfill\(\Box\)
\end{definition}
 Below, we give a family of parametric distributions where $MRLAI$ order between the random variables is present. 
\begin{eg}
Let \(X_i\), \(i=1,2\), have Pareto type II survival function
\[
\bar F_{X_i}(t)=\left(1+\frac{t}{\beta_i}\right)^{-\alpha_i},
\qquad \alpha_i>1,\ \beta_i>0,\ t\geq0 .
\]
If \(\beta_1\geq \beta_2\), then \(X_1\leq_{MRLAI}X_2\). \hfill\(\Box\)
\end{eg}
Below we give an equivalent condition for $MRLAI$ order. The proof being easy is omitted.
\begin{theorem}\label{thm:mrlai_order_equiv}
For two lifetime random variables \(X\) and \(Y\),
\(X\leq_{MRLAI}Y\) if and only if
\(A_X(t)/A_Y(t)\) is nonincreasing on \((0,\tau_{XY})\).
\end{theorem}

The following result gives a simple sufficient condition for the \(MRLAI\) order in terms of the ratio of the two \(MRL\) functions.

\begin{theorem}\label{thm:mrl_ratio_implies_mrlai}
Let \(X\) and \(Y\) have positive continuous \(MRL\) functions on \((0,\tau_X)\) and \((0,\tau_Y)\), respectively. If \(\mu_X(t)/\mu_Y(t)\) is nonincreasing on \((0,\tau_{XY})\), then \(X\leq_{MRLAI}Y\).
\end{theorem}

\begin{proof}
Put \(q(t)=\mu_X(t)/\mu_Y(t)\). Then
\[
\frac{A_X(t)}{A_Y(t)}
=
\frac{\int_0^t q(u)\mu_Y(u)\,du}{\int_0^t\mu_Y(u)\,du}.
\]
Since \(q\) is nonincreasing, \(q(t)\leq A_X(t)/A_Y(t)\). Hence
\[
\frac{d}{dt}\left\{\frac{A_X(t)}{A_Y(t)}\right\}
=
\frac{\mu_Y(t)}{A_Y(t)}
\left\{
q(t)-\frac{A_X(t)}{A_Y(t)}
\right\}
\leq 0.
\]
Thus, \(A_X(t)/A_Y(t)\) is nonincreasing on \((0,\tau_{XY})\), and therefore \(X\leq_{MRLAI}Y\) by Theorem~\ref{thm:mrlai_order_equiv}.
\end{proof}

The next result gives the residual-lifetime version of the \(MRLAI\) order. For \(0<a<\tau_{XY}\), let \(X_a=(X-a\mid X>a)\) and \(Y_a=(Y-a\mid Y>a)\) denote the residual lifetimes of \(X\) and \(Y\) at age \(a\).

\begin{theorem}\label{thm:residual_mrlai_order}
Let \(X\) and \(Y\) have positive continuous \(MRL\) functions on \((0,\tau_X)\) and \((0,\tau_Y)\), respectively. Then, for each fixed \(0<a<\tau_{XY}\), \(X_a\leq_{MRLAI}Y_a\) if and only if
\[
\frac{A_X(a+s)-A_X(a)}
     {A_Y(a+s)-A_Y(a)}
\]
is nonincreasing on \((0,\tau_{XY}-a)\).
\end{theorem}

\begin{proof}
For \(0<s<\tau_{XY}-a\), the residual-lifetime identity gives
\[
\mu_{X_a}(s)=\mu_X(a+s),\qquad
\mu_{Y_a}(s)=\mu_Y(a+s).
\]
Therefore,
\[
\int_0^s\mu_{X_a}(u)\,du
=
\int_a^{a+s}\mu_X(v)\,dv
=
A_X(a+s)-A_X(a),
\]
and similarly,
\[
\int_0^s\mu_{Y_a}(u)\,du
=
A_Y(a+s)-A_Y(a).
\]
By the characterisation of the \(MRLAI\) order in Theorem~\ref{thm:mrlai_order_equiv}, \(X_a\leq_{MRLAI}Y_a\) if and only if
\[
\frac{\displaystyle\int_0^s\mu_{X_a}(u)\,du}
     {\displaystyle\int_0^s\mu_{Y_a}(u)\,du}
\]
is nonincreasing on \((0,\tau_{XY}-a)\). Substituting the preceding expressions gives the required condition. This completes the proof.
\end{proof}

The following consequence gives a simple condition under which the \(MRLAI\) order is preserved for all residual lifetimes.

\begin{corollary}\label{cor:residual_mrlai_order_sufficient}
Let \(X\) and \(Y\) have positive continuous \(MRL\) functions on \((0,\tau_X)\) and \((0,\tau_Y)\), respectively. If \(\mu_X(t)/\mu_Y(t)\) is nonincreasing on \((0,\tau_{XY})\), then \(X_a\leq_{MRLAI}Y_a\) for every \(0<a<\tau_{XY}\).
\end{corollary}

\begin{proof}
Fix \(0<a<\tau_{XY}\) and put \(q(t)=\mu_X(t)/\mu_Y(t)\). Since \(q\) is nonincreasing, for \(0<s<\tau_{XY}-a\),
\[
q(a+s)
\leq
\frac{\displaystyle\int_a^{a+s}q(u)\mu_Y(u)\,du}
     {\displaystyle\int_a^{a+s}\mu_Y(u)\,du}
=
\frac{A_X(a+s)-A_X(a)}
     {A_Y(a+s)-A_Y(a)}.
\]
Hence
\[
\frac{d}{ds}
\left\{
\frac{A_X(a+s)-A_X(a)}
     {A_Y(a+s)-A_Y(a)}
\right\}
=
\frac{\mu_Y(a+s)}
     {A_Y(a+s)-A_Y(a)}
\left[
q(a+s)
-
\frac{A_X(a+s)-A_X(a)}
     {A_Y(a+s)-A_Y(a)}
\right]
\leq 0.
\]
Thus, the ratio
\[
\frac{A_X(a+s)-A_X(a)}
     {A_Y(a+s)-A_Y(a)}
\]
is nonincreasing on \((0,\tau_{XY}-a)\). Therefore, by Theorem~\ref{thm:residual_mrlai_order}, \(X_a\leq_{MRLAI}Y_a\). Since \(a\in(0,\tau_{XY})\) was arbitrary, the result follows.
\end{proof}

The condition in Corollary~\ref{cor:residual_mrlai_order_sufficient} is only sufficient. The following counterexample shows that the \(MRLAI\) order is not preserved under residual-lifetime transformation in general.

\begin{example}\label{cex:mrlai_residual_not_preserved}
Let \(Y\) be an exponential random variable with mean one. Then \(\mu_Y(t)=1\) and \(L_Y^\mu(t)=1\) for all \(t>0\). Let \(X\) have the \(MRL\) function
\[
\mu_X(t)=
\begin{cases}
1-0.4t, & 0\leq t<2,\\
0.1t, & 2\leq t<3,\\
0.3, & t\geq3.
\end{cases}
\]
The corresponding \(MRLAI\) function is obtained as
\[
L_X^\mu(t)=
\begin{cases}
\dfrac{1-0.4t}{1-0.2t}, & 0<t<2,\\[6pt]
\dfrac{0.1t^2}{1+0.05t^2}, & 2\leq t<3,\\[6pt]
\dfrac{0.3t}{0.3t+0.55}, & t\geq3.
\end{cases}
\]
Hence \(L_X^\mu(t)<1=L_Y^\mu(t)\) for all \(t>0\), and therefore \(X\leq_{MRLAI}Y\). Now consider the residual lifetimes at age \(a=2\). 
Since \(Y\) is exponential, \(Y_2\) is again exponential with mean one, and hence \(L_{Y_2}^\mu(s)=1\) for all \(s>0\). For \(0<s<1\),
\[
\mu_{X_2}(s)=\mu_X(2+s)=0.1(2+s).
\]
Thus
\[
L_{X_2}^\mu(s)
=
\frac{s\,0.1(2+s)}{\int_0^s0.1(2+u)\,du}
=
\frac{4+2s}{4+s}
>1
=
L_{Y_2}^\mu(s),\qquad 0<s<1.
\]

Consequently, \(X_2\not\leq_{MRLAI}Y_2\), although \(X\leq_{MRLAI}Y\). Therefore, the \(MRLAI\) order is not preserved under residual-lifetime transformation in general.
\end{example}

\noindent On any class of lifetime distributions with a common right endpoint, it is easy to see that the \(MRLAI\) order is reflexive (\(i.e.\), \(X\leqslant_{MRLAI}X\)) and transitive (\(i.e.\), if \(X\leqslant_{MRLAI}Y\) and \(Y\leqslant_{MRLAI}Z\), then \(X\leqslant_{MRLAI}Z\)). To see whether the \(MRLAI\) order is antisymmetric, note that if \(X\leqslant_{MRLAI}Y\) and \(Y\leqslant_{MRLAI}X\), then \(L_X^\mu(t)=L_Y^\mu(t)\) for all \(0<t<\tau_{XY}\), which means
\[
\frac{\mu_X(t)}{\int_0^t\mu_X(u)\,du}
=
\frac{\mu_Y(t)}{\int_0^t\mu_Y(u)\,du}.
\]
Writing \(G_X(t)=\int_0^t\mu_X(u)\,du\), \(G_Y(t)=\int_0^t\mu_Y(u)\,du\), and \(h(t)=G_X(t)/G_Y(t)\), the above equation can be rewritten as \(h'(t)=0\), \(i.e.\), \(h(t)={\rm constant}\). This gives \(\mu_X(t)=c\,\mu_Y(t)\), for some \(c>0\) and all \(0<t<\tau_{XY}\), so that \(X\) and \(Y\) have proportional \(MRL\) functions. Thus, on any class of lifetime distributions with a common right endpoint, the \(MRLAI\) order is a preorder and need not be antisymmetric. If \(\mu_X(0+)=\mu_Y(0+)\), then \(c=1\), and hence \(\mu_X(t)=\mu_Y(t)\) for all \(0<t<\tau_{XY}\). If, in addition, \(\tau_X=\tau_Y\) and \(\bar F_X(0+)=\bar F_Y(0+)=1\), then the \(MRL\) function determines the distribution, and \(X\) and \(Y\) are identically distributed. Thus, the \(MRLAI\) order is antisymmetric on any class of lifetime distributions with a common right endpoint, common initial \(MRL\), and no atom at the origin.

\hspace*{.2in} We next discuss the relationship between the \(MRLAI\) order and existing stochastic orders. For this, we need to recall the definitions of the following well-known orderings from the literature (see Shaked and Shanthikumar~\cite{shaked2007stochastic}).
\begin{definition}\label{2}
    A nonnegative absolutely continuous random variable \( X \) is considered smaller than another nonnegative absolutely continuous random variable \( Y \) in
\begin{itemize}
  \item[(i)] likelihood ratio order, written as \( X \leq_{lr} Y \),
  if the ratio of their densities \( f_Y(t)/f_X(t) \) is nondecreasing for all \(0<t<\tau_{XY}\) for which the ratio is defined;
  \item[(ii)] hazard rate order, denoted by \(X\leq_{hr}Y\),
  if their failure rate functions satisfy \( r_X(t) \geq r_Y(t) \) for every \(0<t<\tau_{XY}\);
  \item[(iii)] usual stochastic order, denoted by \( X \leq_{st} Y \),
  if the survival function of $X$ is dominated by that of $Y$, i.e., \( \bar{F}_X(t) \leq \bar{F}_Y(t) \) for all \( t \geq 0 \);
  \item[(iv)] increasing convex order, written as \( X \leq_{icx} Y \), if
  \[
  \int_t^{\infty} \bar{F}_X(u) \, du \leq \int_t^{\infty} \bar{F}_Y(u) \, du,
  \]
  for all \(t \geq 0\);
  \item[(v)] mean residual life order, written as \( X \leq_{mrl} Y \), if
  \[
  \frac{\int_t^{\infty} \bar{F}_X(u) \, du}{\bar{F}_X(t)}
  \leq
  \frac{\int_t^{\infty} \bar{F}_Y(u) \, du}{\bar{F}_Y(t)},
  \]
  for all \(0<t<\tau_{XY}\);
  \item[(vi)] variance residual life order, written as \( X \leq_{vrl} Y \), if
\[
\frac{\displaystyle\int_t^{\infty} \int_u^{\infty} \bar{F}_Y(v) \, dv \, du}
{\displaystyle\int_t^{\infty} \int_u^{\infty} \bar{F}_X(v) \, dv \, du}
\quad\text{is nondecreasing for }0<t<\tau_{XY}.
\]
\end{itemize}
\end{definition}
The standard implication relationships among these stochastic orders are summarised in Figure~\ref{fig:stochastic_orders}.
\begin{figure}[ht]
\centering
\begin{tikzpicture}[
    node distance=1.4cm,
    every node/.style={
        draw,
        rectangle,
        rounded corners,
        align=center,
        minimum width=4.2cm,
        minimum height=1.0cm
    },
    arrow/.style={->, thick}
]

\node (lr) {Likelihood Ratio Order\\ \((\leq_{lr})\)};
\node (hr) [below=of lr] {Hazard Rate Order\\ \((\leq_{hr})\)};
\node (mrl) [below left=of hr, xshift=-0.7cm] {Mean Residual Life Order\\ \((\leq_{mrl})\)};
\node (st)  [below right=of hr, xshift=0.7cm] {Usual Stochastic Order\\ \((\leq_{st})\)};
\node (vrl) [below=of mrl] {Variance Residual Life Order\\ \((\leq_{vrl})\)};
\node (icx) [below=of st] {Increasing Convex Order\\ \((\leq_{icx})\)};

\draw[arrow] (lr) -- (hr);
\draw[arrow] (hr) -- (mrl);
\draw[arrow] (hr) -- (st);
\draw[arrow] (mrl) -- (vrl);
\draw[arrow] (mrl) -- (icx);
\draw[arrow] (st) -- (icx);

\end{tikzpicture}
\caption{Implication diagram for the stochastic orders considered in Definition~\ref{2}.}
\label{fig:stochastic_orders}
\end{figure}

The following counterexample shows that $MRLAI$ ordering does not imply increasing convex ordering.
\begin{example}\label{ex:mrlai-not-imply-icx}
    Let $X$ be an exponential random variable having failure rate $0.5$, and let $Y$ be another random variable having survival function given by $\bar{F}_{Y}(t) = (t+1)^{-3}$, for $t \geq 0$. Then, we have $L_{X}^{\mu}(t) = 1$ and $L_{Y}^{\mu}(t) =2(t+1)/(t+2)$, for $t \geq 0$, so that $X \leq_{MRLAI}Y$. However, $X \centernot \leq_{icx} Y$ since 
    \begin{eqnarray*}
        \int_{t}^{\infty} \bar{F}_{X}(u) du-\int_{t}^{\infty} \bar{F}_{Y}(u) du&=&2e^{-t/2}-\frac{1}{2 \left(t + 1\right)^{2}}\\
        &=&\left\{\begin{array}{rl}
       0.6802033,&{\rm if}\;t=2\\
       -0.001138222,&{\rm if}\;t= 18
        \end{array}\right.
    \end{eqnarray*} 
    Moreover, \(\mu_X(1)=2>\mu_Y(1)=1\), and hence \(X\centernot\leq_{mrl}Y\).
\end{example}
\noindent The following counterexample shows that $MRLAI$ order does not imply $vrl$ order.
\begin{example}
Let \(X\) have survival function \(\bar F_X(t)=(t^2+2t+2)e^{-t}/2\), and let \(Y\) have Pareto survival function \(\bar F_Y(t)=(t+1)^{-3}\), \(t\geq0\). Then
\[
\mu_X(u)=\frac{u^2+4u+6}{u^2+2u+2},
\quad
\int_0^t\mu_X(u)\,du
=t+\ln(t^2+2t+2)-\ln2+2\tan^{-1}(t+1)-\frac{\pi}{2},
\]
so that
\[
L_X^\mu(t)=
\frac{t(t^2+4t+6)}
{(t^2+2t+2)\left\{\ln(t^2+2t+2)+2\tan^{-1}(t+1)+t-\ln2-\pi/2\right\}}.
\]
Similarly, \(\mu_Y(u)=(u+1)/2\), and hence \(L_Y^\mu(t)=2(t+1)/(t+2)\). It can be verified that \(L_X^\mu(t)\leq L_Y^\mu(t)\) for all \(t>0\), so that \(X\leq_{MRLAI}Y\). Next, define
\[
b(t)=
\frac{\displaystyle\int_t^\infty\int_u^\infty \bar F_X(v)\,dv\,du}
{\displaystyle\int_t^\infty\int_u^\infty \bar F_Y(v)\,dv\,du}
=(t+1)(t^2+6t+12)e^{-t}.
\]
Then \(b(0.40)=13.66380\), \(b(0.80)=14.10533\), and \(b(1.5)=12.96944\). Since the ratio required for the \(vrl\) order is \(1/b(t)\), and \(1/b(0.40)>1/b(0.80)\), we have \(X\centernot\leq_{vrl}Y\). Therefore, \(X\leq_{MRLAI}Y\) does not imply \(X\leq_{vrl}Y\). \hfill\(\Box\)
\end{example}
The converse implication also fails, as shown by the following counterexample. 
\begin{example}
Let \(X\) and \(Y\) have Erlang densities \(f_X(t)=9te^{-3t}\) and \(f_Y(t)=4te^{-2t}\), \(t\ge0\), respectively. Then \(\bar F_X(t)=(1+3t)e^{-3t}\) and \(\bar F_Y(t)=(1+2t)e^{-2t}\). A direct calculation gives
\[
\frac{\displaystyle\int_t^\infty\int_u^\infty \bar F_Y(v)\,dv\,du}
{\displaystyle\int_t^\infty\int_u^\infty \bar F_X(v)\,dv\,du}
=
\frac{3e^t(2t+3)}{4(t+1)},
\]
whose logarithmic derivative is \(1-\{(2t+3)(t+1)\}^{-1}>0\). Hence \(X\leq_{vrl}Y\). On the other hand, \(\mu_X(t)=(3t+2)/\{3(3t+1)\}\) and \(\mu_Y(t)=(t+1)/(2t+1)\), so that
\[
R(t):=
\frac{\displaystyle\int_0^t\mu_X(u)\,du}
{\displaystyle\int_0^t\mu_Y(u)\,du}
=
\frac{4\{3t+\ln(3t+1)\}}{9\{2t+\ln(2t+1)\}}.
\]
Since \(R(0.1)=0.653742>R(1.5)=0.628701<R(5)=0.637118\), \(R(t)\) is not nonincreasing. Therefore \(X\centernot\leq_{MRLAI}Y\), showing that \(X\leq_{vrl}Y\centernot\implies X\leq_{MRLAI}Y\). \hfill\(\Box\)
\end{example}
\hspace*{.2in} The following counterexample shows that likelihood ratio ordering does not imply $MRLAI$ ordering. 
\begin{example}\label{ex:lr-not-imply-mrlai}
    Let $X$ and $Y$ be two independent random variables having Erlang distribution with respective density functions $f_{X}(t) = 9t\mathrm{e}^{-3t}, t\geq0$
and 
$f_{Y}(t) =4te^{-2t} , t\geq 0.$
It is easy to see that $X\leqslant_{lr}Y$. However, 
\begin{eqnarray*}
    \frac{\int_0^t\mu_X(u)\,du}{\int_0^t\mu_Y(u)\,du}&=&\frac{4\{3t+\ln(3t+1)\}}{9\{2t+\ln(2t+1)\}}.\\
    &=&\left\{\begin{array}{ll}
        0.65370,&{\rm if}\;t=0.1\\
        0.6287006,&{\rm if}\; t=1.5\\
        0.6371185,&{\rm if}\; t=5
    \end{array}\right.
\end{eqnarray*}
Thus, we have  \[
X\leq_{lr}Y \centernot\implies X\leq_{MRLAI}Y.
\]
\end{example}
Since likelihood ratio ordering does not imply $MRLAI$ ordering, it can be concluded that none of $hr$, $st$, $mrl$, $vrl$ and $icx$ orders implies the $MRLAI$ ordering.

Counterexamples~\ref{ex:mrlai-not-imply-icx} and~\ref{ex:lr-not-imply-mrlai} also show that the \(MRLAI\) order is not equivalent to the usual \(MRL\) order. Indeed, Counterexample~\ref{ex:mrlai-not-imply-icx} gives \(X\leq_{MRLAI}Y\) but \(X\centernot\leq_{mrl}Y\), whereas Counterexample~\ref{ex:lr-not-imply-mrlai} gives \(X\leq_{lr}Y\), and hence \(X\leq_{mrl}Y\), but \(X\centernot\leq_{MRLAI}Y\).
We next give a sufficient condition for the \(MRLAI\) order.

\begin{theorem}\label{thm:dmrla-imrla-mrlai-order}
Let \(X\) and \(Y\) be lifetime random variables. If \(X\) is \(DMRLA\) and \(Y\) is \(IMRLA\), then \(X\leq_{MRLAI}Y\).
\end{theorem}

\begin{proof}
By Theorem~\ref{thm:mrlai_average_mrl}, \(X\in DMRLA\) implies
\(L_X^\mu(t)\leq1\), whereas \(Y\in IMRLA\) implies
\(L_Y^\mu(t)\geq1\). Hence
\[
L_X^\mu(t)\leq L_Y^\mu(t),
\qquad 0<t<\tau_{XY}.
\]
and therefore \(X\leq_{MRLAI}Y\).
\end{proof}

The preceding theorem provides only a sufficient condition for \(X\leq_{MRLAI}Y\); the condition is not necessary, as the following counterexample shows.
\begin{example}
Let \(X\) and \(Y\) have \(MRL\) functions \(\mu_X(t)=1+t\) and \(\mu_Y(t)=1+3t\), respectively, for \(t\geq0\). Then \(L_X^\mu(t)=2(t+1)/(t+2)\) and \(L_Y^\mu(t)=2(1+3t)/(2+3t)\), \(t>0\). Since \(L_Y^\mu(t)-L_X^\mu(t)=4t/\{(t+2)(3t+2)\}>0\) for all \(t>0\), we have \(X\leq_{MRLAI}Y\). However, both \(X\) and \(Y\) are \(IMRLA\), showing that the sufficient condition in the preceding theorem is not necessary. \hfill\(\Box\)
\end{example}
Since \(DMRL\) implies \(DMRLA\) and \(IMRL\) implies \(IMRLA\), Theorem~\ref{thm:dmrla-imrla-mrlai-order} yields the following corollary.
\begin{corollary}
  If $X$ is  decreasing in $MRL$  and $Y$ is increasing in $MRL$ then $X \leq_{MRLAI}Y$ .  
\end{corollary}
We next provide conditions under which the \(MRLAI\) order is preserved for series systems formed from two sets of components.
\begin{theorem}
Let $\{X_i,\, i=1,\dots,m\}$ and $\{Y_j,\, j=1,\dots,n\}$ be two sets of independent and identically distributed random variables with $MRL$ functions  $\mu_{X_i}(t)=a+bt$ and $\mu_{Y_j}(t)=c+dt$, respectively, where $a,c>0$ and $b,d\ge 0$.
Then the following statements hold.
\begin{itemize}
\item[(i)] $X_i \le_{\mathrm{MRLAI}} Y_j$ if and only if $ad-bc\ge 0$.
\item[(ii)] If $X_i \le_{\mathrm{MRLAI}} Y_j$, then for all $m,n\ge 1$,
\[
X_{(m)}=\min_{1\le i\le m}X_i
\le_{\mathrm{MRLAI}}
Y_{(n)}=\min_{1\le j\le n}Y_j.
\]
\end{itemize}
\end{theorem}

\begin{proof}

(i) For $t>0$, 
\begin{eqnarray*}
L_{X_i}^{\mu}(t)
&=&\frac{a+bt}{a+\frac{b}{2}t}, \\
L_{Y_j}^{\mu}(t)&=&
\frac{c+dt}{c+\frac{d}{2}t}.
\end{eqnarray*}
A simple calculation shows that 
\(
L_{X_i}^{\mu}(t)\le L_{Y_j}^{\mu}(t)
\) if and only if 
\(
ad-bc\ge 0.
\)
This proves part~ (i).

\medskip

To prove (ii), note that if $b>0$, the survival function corresponding to $\mu_{X_i}(t)=a+bt$ is given by
\[
\bar F_{X_i}(t)
=\left(1+\frac{b}{a}t\right)^{-1-\frac{1}{b}}
\]
so that 
\[
\bar F_{X_{(m)}}(t)
=\bigl(\bar F_{X_i}(t)\bigr)^m
=\left(1+\frac{b}{a}t\right)^{-m\left(1+\frac{1}{b}\right)}.
\]
The corresponding $MRL$ function is given by 
\[
\mu_{X_{(m)}}(t)
=\frac{a+bt}{m(b+1)-b}.
\]
This gives 
\[
L_{X_{(m)}}^{\mu}(t)
=\frac{a+bt}{a+\frac{b}{2}t}
=L_{X_i}^{\mu}(t).
\]
If $b=0$, then $X_i$ is exponentially distributed with failure rate $1/a$, and
$X_{(m)}$ is exponential with failure rate $m/a$. In this case, the $MRL$ function is constant and equal to $a/m$, and hence
$L_{X_{(m)}}^{\mu}(t)=1=L_{X_i}^{\mu}(t).$ An entirely analogous argument applies to $Y_{(n)}$, yielding
$L_{Y_{(n)}}^{\mu}(t)=L_{Y_j}^{\mu}(t).$ Therefore, if $ad-bc\ge 0$, then for all $t>0$,
\[
L_{X_{(m)}}^{\mu}(t)
=L_{X_i}^{\mu}(t)
\le
L_{Y_j}^{\mu}(t)
=L_{Y_{(n)}}^{\mu}(t),
\]
which proves that
$X_{(m)} \le_{\mathrm{MRLAI}} Y_{(n)}.$
\end{proof}

 In the following counterexample we show that the $MRLAI$ ordering is not closed under the formation of a parallel system.
 \begin{example}
Let $X_{1}$ and $X_{2}$ be two $i.i.d.$ random variables having Erlang distribution with survival function 
     \(\bar{F}_{X_{i}}(t) = (1 +t)e^{-t}, \;t\geq 0\)
   Also, let $Y_{1}$ and $Y_{2}$ be two $i.i.d.$ random variables having exponential distribution with survival function
\(
\bar{F}_{Y_i}(t) = e^{-2t}, \quad t \ge 0.
\)
The $MRLAI$ function of $X_i$ is given by
$$L_{X_i}^{\mu}(t) 
= \frac{t\mu_{X_i}(t)}{\int_0^t \mu_{X_i}(u)\,du} = \frac{t(t+2)}{(t+1)(\ln(t+1)+t)} < 1,$$
whereas, for $Y_i$, it is unity, $i.e.$,
$L_{Y_i}^{\mu}(t)=1,$
for all $t>0$ and $i,j = 1,2$. Thus,
\[
X_i \leq_{MRLAI} Y_j, \quad \text{for all } i,j=1,2.
\]
If we write 
\(
X = \max\{X_1, X_2\}\) and \(Y = \max\{Y_1, Y_2\},
\)
then the $MRL$ function of $X$ is given by
\[
\begin{aligned}
\mu_X(t)
&= \frac{2\left( 4(t+2)e^t - t(t+3) \right) - 5}{4(t+1)\left( 2e^t - t - 1 \right)},
\end{aligned}
\]
and the corresponding $MRLAI$ is given by \[
L_X^{\mu}(t)
=
\frac{
2\left( 4(t+2)e^t - t(t+3) \right) - 5
}{
4(t+1)\left( 2e^t - t - 1 \right)
}
\left/
\left(
\frac{1}{t}
\int_0^t
\frac{2\left( 4(u+2)e^u - u(u+3) \right) - 5}
{4(u+1)\left( 2e^u - u - 1 \right)}
\,du
\right)
\right..
\]
 Similarly, the $MRL$ function of $Y$ is
$\mu_Y(t)=(4 e^{2t} - 1)/(8 e^{2t} - 4).$
Hence the corresponding $MRLAI$ function of $Y$ is given by 
\[
\begin{aligned}
L_Y^\mu(t) 
&= \frac{8 t (4 e^{2t} - 1)}{(8 e^{2t} - 4) \left( \ln(16 e^{4t} - 8 e^{2t}) - \ln(8) \right)}.
\end{aligned}
\]
Clearly, from Figure~\ref{fig:MRLAI_XY}, the curves of $L_X^{\mu}(t)$ and $L_Y^{\mu}(t)$ intersect. Therefore, the $MRLAI$ order is not preserved under the formation of a parallel system. \hfill$\Box$

  \begin{figure}[htbp]
    \centering
    \includegraphics[width=0.60 \textwidth]{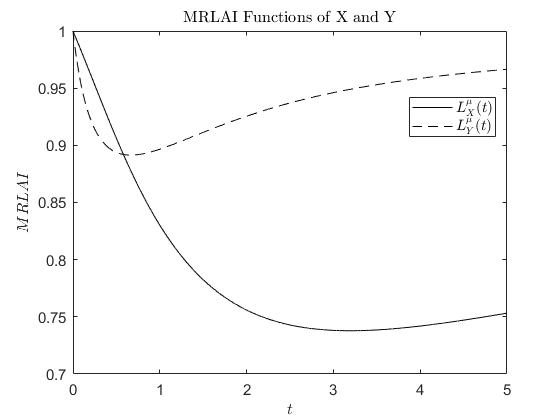}
    \caption{MRLAI functions of the random variables $X$ and $Y$. }
    \label{fig:MRLAI_XY}
\end{figure}
 \end{example}
 \noindent This shows that the said order is not closed under the formation of a coherent system.
 
 The next result shows that the \(MRLAI\) order is preserved under positive scale transformations. 
\begin{theorem}\label{thm:5.5}
Let $X$ and $Y$ be non-negative continuous random variables such that
$X \le_{MRLAI} Y$. Then, for every $a>0$, the $MRLAI$ order is preserved under
the transformation $\phi(x)=ax$, that is,
\[
\phi(X) \le_{MRLAI} \phi(Y).
\]
\end{theorem}

\begin{proof}
Let $U=\phi(X)=aX$ and $V=\phi(Y)=aY$, where $a>0$. The survival functions of $U$ and $V$ satisfy, for all $u \ge 0$,
\[
\bar F_U(u)
=\bar F_X\!\left(\frac{u}{a}\right)
\qquad
{\rm and}\qquad
\bar F_V(u)
=\bar F_Y\!\left(\frac{u}{a}\right).
\]
The $MRL$ function of $U$ comes out to be
\[
\mu_U(u)
=a\,\mu_X\!\left(\frac{u}{a}\right),
\qquad 0<u<a\tau_X.
\]
This, after simple algebra, gives
\[
L_U^\mu(u)
=L_X^\mu\!\left(\frac{u}{a}\right)
\qquad
{\rm and}\qquad
L_V^\mu(u)
=L_Y^\mu\!\left(\frac{u}{a}\right),
\qquad 0<u<a\tau_{XY}.
\]
Since $X\le_{MRLAI} Y$, it follows that
\[
L_X^\mu\!\left(\frac{u}{a}\right)
\le
L_Y^\mu\!\left(\frac{u}{a}\right),
\qquad 0<u<a\tau_{XY}.
\]
Consequently,
\[
L_U^\mu(u)\le L_V^\mu(u),
\qquad 0<u<a\tau_{XY}.
\]
Hence, $\phi(X)\le_{MRLAI}\phi(Y)$, which completes the proof.
\end{proof}
Further, we illustrate that the $MRLAI$ order is not necessarily preserved under a strictly increasing affine transformation $\phi(x) = x + b$, $b>0$.

\begin{example}\label{5.6}
Let $X$ and $Y$ be two non-negative random variables with \(\bar F_X(0)=\bar F_Y(0)=1\) and $MRL$ functions
\(
\mu_X(t) = 10 + t\), and \(\mu_Y(t) = 1 + 0.2 t\), for \(t \ge 0,
\)
so that their corresponding MRLAI functions are
\[
L_X^\mu(t) = \frac{10 + t}{10 + 0.5 t}, \qquad 
L_Y^\mu(t) = \frac{1 + 0.2 t}{1 + 0.1 t}.
\]
It can be verified algebraically that $L_X^\mu(t) < L_Y^\mu(t)$ for all $t>0$, showing that $X \le_{\mathrm{MRLAI}} Y$. Now, consider the affine transformation $\phi(x) = x + 2$ and define
\(
U = X + 2\), and \( V = Y + 2.
\)
For the transformed variables, the $MRL$ functions are given by
\[
\mu_U(s) = \frac{\int_s^\infty \bar{F}_U(u)\, du}{\bar{F}_U(s)}, \qquad 
\bar{F}_U(u) =
\begin{cases}
1, & 0\leq u<2,\\
\bar F_X(u-2), & u\geq2,
\end{cases}
\]
\[
\mu_V(s) = \frac{\int_s^\infty \bar{F}_V(u)\, du}{\bar{F}_V(s)}, \qquad 
\bar{F}_V(u) =
\begin{cases}
1, & 0\leq u<2,\\
\bar F_Y(u-2), & u\geq2.
\end{cases}
\]
For $0 \le s \le 2$ we have $s-2 \le 0$, and since $X \ge 0$ and $Y \ge 0$, it follows that
\[
\mu_U(s)= 12 - s, \qquad \mu_V(s) =  3 - s.
\]
For $s>2$, the $MRL$ functions shift by 2 units as
\[
\mu_U(s) = \mu_X(s-2) = 8 + s, \qquad \mu_V(s) = \mu_Y(s-2) = 0.6 + 0.2 s.
\]
Writing \(I_U(t)=\int_0^t\mu_U(s)\,ds\) and \(I_V(t)=\int_0^t\mu_V(s)\,ds\), we obtain
\[
I_U(t)=
\begin{cases}
12t-\dfrac{t^2}{2}, & 0<t\leq2,\\[5pt]
\dfrac{t^2}{2}+8t+4, & t>2,
\end{cases}
\qquad
I_V(t)=
\begin{cases}
3t-\dfrac{t^2}{2}, & 0<t\leq2,\\[5pt]
\dfrac{t^2}{10}+\dfrac{6t}{10}+2.4, & t>2.
\end{cases}
\]
Hence the $MRLAI$ functions of $U$ and $V$ become
\[
L_U^\mu(t)=
\begin{cases}
\dfrac{t(12-t)}{12t-\frac12t^2}, & 0<t\leq2,\\[8pt]
\dfrac{t(8+t)}{\frac12t^2+8t+4}, & t>2,
\end{cases}
\qquad
L_V^\mu(t)=
\begin{cases}
\dfrac{t(3-t)}{3t-\frac12t^2}, & 0<t\leq2,\\[8pt]
\dfrac{t(0.6+0.2t)}{0.1t^2+0.6t+2.4}, & t>2.
\end{cases}
\]
The two functions cross at \(t=(8+2\sqrt{226})/5\), as illustrated in Figure~\ref{fig:mrlaicounterexample}. Therefore, the order is not preserved under the affine shift $\phi(x) = x + 2$.\hfill$\Box$

\begin{figure}[htbp]
    \centering
    \includegraphics[width=0.7\textwidth]{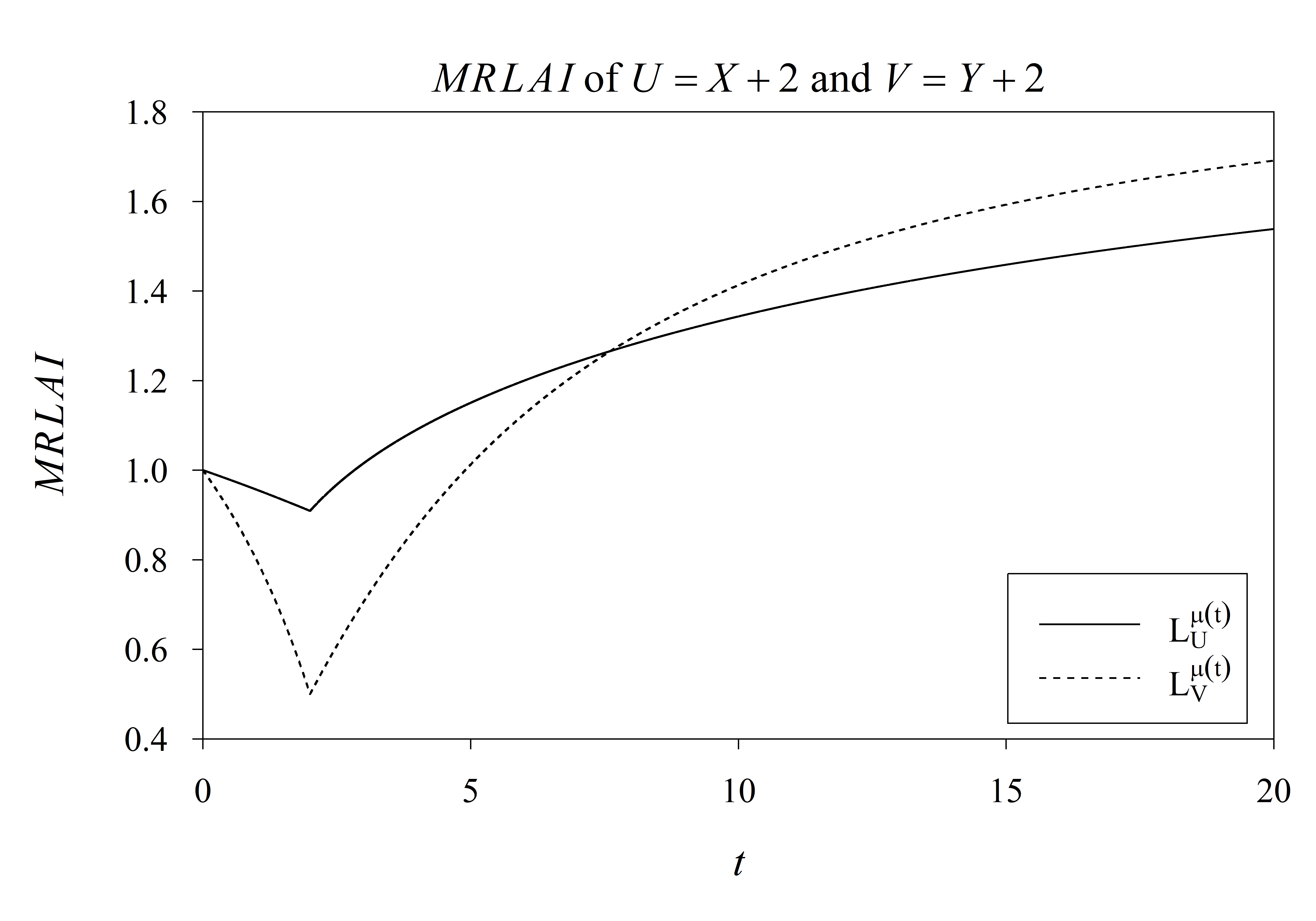}
    \caption{The $MRLAI$ order $X \leq_{\mathrm{MRLAI}} Y$ is not preserved under the shift $U=X+2$, $V=Y+2$.}
    \label{fig:mrlaicounterexample}
\end{figure}
\end{example}
Counterexample~\ref{5.6} shows that the preservation result in Theorem~\ref{thm:5.5} does not extend to affine shifts.

\section{Conclusion}

This paper introduced the mean residual life ageing intensity (\(MRLAI\)) function for studying the ageing behaviour of a lifetime distribution conditional on survival up to a given time \(0<t<\tau_X\). The proposed function is an \(MRL\)-based analogue of the failure-rate-based ageing intensity function and compares the current \(MRL\) with its average over \([0,t]\). Several basic properties of the \(MRLAI\) function were established, and the associated non-parametric ageing classes \(DMRLAI\) and \(IMRLAI\) were studied in relation to the usual \(MRL\)-based ageing classes.

The closure behaviour of the \(DMRLAI\) and \(IMRLAI\) classes was examined under standard reliability operations, including mixtures, convolutions and \(k\)-out-of-\(n\) systems. In addition, an \(MRLAI\)-based stochastic order was introduced and compared with several classical stochastic orders. The examples and counterexamples show that the \(MRLAI\) function and the associated order provide a distinct framework for describing residual-life ageing behaviour.

\bibliographystyle{acm}
\bibliography{bib}

@article{jiang2003aging,
  title={Aging property of unimodal failure rate models},
  author={Jiang, R and Ji, Ping and Xiao, X},
  journal={Reliability Engineering \& System Safety},
  volume={79},
  number={1},
  pages={113--116},
  year={2003},
  publisher={Elsevier}
}

@article{nanda2007properties,
  title={Properties of aging intensity function},
  author={Nanda, Asok K and Bhattacharjee, Subarna and Alam, SS},
  journal={Statistics \& probability letters},
  volume={77},
  number={4},
  pages={365--373},
  year={2007},
  publisher={Elsevier}
}

@article{bhattacharjee2013reliability,
  title={Reliability analysis using ageing intensity function},
  author={Bhattacharjee, Subarna and Nanda, Asok K and Misra, Satya Kr},
  journal={Statistics \& Probability Letters},
  volume={83},
  number={5},
  pages={1364--1371},
  year={2013},
  publisher={Elsevier}
}

@article{szymkowiak2018characterizations,
  title={Characterizations of distributions through aging intensity},
  author={Szymkowiak, Magdalena},
  journal={IEEE Transactions on Reliability},
  volume={67},
  number={2},
  pages={446--458},
  year={2018},
  publisher={IEEE}
}

@article{sunoj2018quantile,
  title={A quantile-based study on ageing intensity function},
  author={Sunoj, SM and Rasin, RS},
  journal={Communications in Statistics-Theory and Methods},
  volume={47},
  number={22},
  pages={5474--5484},
  year={2018},
  publisher={Taylor \& Francis}
}

@article{buono2021generalized,
  title={On generalized reversed aging intensity functions},
  author={Buono, Francesco and Longobardi, Maria and Szymkowiak, Magdalena},
  journal={Ricerche di matematica},
  pages={1--24},
  year={2021},
  publisher={Springer}
}

@article{buono2022multivariate,
  title={Multivariate conditional aging intensity functions and load-sharing models},
  author={Buono, Francesco},
  journal={Hacettepe Journal of Mathematics and Statistics},
  pages={1--13},
  year={2022}
}

@article{sunoj2020ageing,
  title={Ageing Intensity Function for Conditionally Specified Models},
  author={Sunoj, SM and Nair, N Unnikrishnan and Nanda, Asok K and Rasin, RS},
  journal={American Journal of Mathematical and Management Sciences},
  volume={39},
  number={4},
  pages={329--344},
  year={2020},
  publisher={Taylor \& Francis}
}

@article{szymkowiak2019measures,
  title={Measures of ageing tendency},
  author={Szymkowiak, Magdalena},
  journal={Journal of Applied Probability},
  volume={56},
  number={2},
  pages={358--383},
  year={2019},
  publisher={Cambridge University Press}
}

@article{bryson1969some,
  title={Some criteria for aging},
  author={Bryson, Maurice C and Siddiqui, MM},
  journal={Journal of the American Statistical Association},
  volume={64},
  number={328},
  pages={1472--1483},
  year={1969},
  publisher={Taylor \& Francis}
}

@book{shaked2007stochastic,
  title={Stochastic orders},
  author={Shaked, Moshe and Shanthikumar, J George},
  year={2007},
  publisher={Springer}
}

@article{deshpande1986aspects,
  title={Aspects of positive ageing},
  author={Deshpande, Jayant V. and Kochar, Subhash C. and Singh, Hari},
  journal={Journal of Applied Probability},
  volume={23},
  number={3},
  pages={748--758},
  year={1986},
  doi={10.2307/3214012},
  publisher={Applied Probability Trust}
}

@incollection{guess1988mean,
  title={Mean residual life: Theory and applications},
  author={Guess, Frank and Proschan, Frank},
  booktitle={Handbook of Statistics},
  volume={7},
  pages={215--224},
  year={1988},
  publisher={Elsevier},
  doi={10.1016/S0169-7161(88)07014-2}
}

@incollection{hall1981mean,
  title={Mean residual life},
  author={Hall, W. J. and Wellner, J. A.},
  booktitle={Statistics and Related Topics},
  editor={Cs{\"o}rg{\H{o}}, M. and Dawson, D. A. and Rao, J. N. K. and Saleh, A. K. Md. E.},
  pages={169--184},
  year={1981},
  publisher={North-Holland},
  address={Amsterdam}
}

@article{hollander1975tests,
  title={Tests for the mean residual life},
  author={Hollander, Myles and Proschan, Frank},
  journal={Biometrika},
  volume={62},
  number={3},
  pages={585--593},
  year={1975},
  doi={10.1093/biomet/62.3.585},
  publisher={Oxford University Press}
}

@article{nanda2010mean,
  title={Mean residual life function, associated orderings and properties},
  author={Nanda, Asok K. and Bhattacharjee, Subarna and Balakrishnan, Narayanaswamy},
  journal={IEEE Transactions on Reliability},
  volume={59},
  number={1},
  pages={55--65},
  year={2010},
  doi={10.1109/TR.2009.2035791},
  publisher={IEEE}
}

@book{barlow1975statistical,
  author    = {Barlow, Richard E. and Proschan, Frank},
  title     = {Statistical Theory of Reliability and Life Testing: Probability Models},
  publisher = {Holt, Rinehart and Winston},
  address   = {New York},
  year      = {1975}
}

@article{newby1986applications,
  author  = {Newby, M. J.},
  title   = {Applications of Concepts of Ageing in Reliability Data Analysis},
  journal = {Reliability Engineering},
  volume  = {14},
  pages   = {291--308},
  year    = {1986},
  doi     = {10.1016/0143-8174(86)90063-6}
}

@article{giri2021ageing,
  author  = {Giri, R. L. and Nanda, A. K. and Dasgupta, M. and Misra, S. K. and Bhattacharjee, S.},
  title   = {On Ageing Intensity Function of Some Weibull Models},
  journal = {Communications in Statistics -- Theory and Methods},
  volume  = {52},
  number  = {1},
  pages   = {227--262},
  year    = {2023},
  doi     = {10.1080/03610926.2021.1910845}
}

@book{bingham1987regular,
  title={Regular Variation},
  author={Bingham, N. H. and Goldie, C. M. and Teugels, J. L.},
  series={Encyclopedia of Mathematics and its Applications},
  volume={27},
  year={1987},
  publisher={Cambridge University Press},
  address={Cambridge},
  doi={10.1017/CBO9780511721434}
}

@article{chen2017elasticity,
  author  = {Chen, Xudong},
  title   = {Elasticity as Relative Slopes: A Graphical Approach to Linking the Concepts of Elasticity and Slope},
  journal = {The American Economist},
  year    = {2017},
  volume  = {62},
  number  = {2},
  pages   = {258--267},
  month   = {October},
  doi     = {10.1177/0569434516682713}
}
\end{document}